\documentclass[12pt, a4paper]{amsart}
\usepackage[utf8]{inputenc}
\usepackage{amsthm}
\usepackage{mathrsfs}  
\usepackage[none]{hyphenat}
\usepackage[english]{babel}
\usepackage{setspace}
\usepackage{longtable}
\usepackage{multicol}
\usepackage{multirow}
\usepackage[a4paper, top=3cm, bottom=4cm,left=3cm,right=3cm]{geometry}
\usepackage{setspace}
\usepackage{comment}
\usepackage{caption}
\usepackage{adjustbox}

\usepackage{pgf,tikz}
\usepackage{microtype}

\usepackage{tikz-cd}
\usepackage{tikz,pgfplots}
\usetikzlibrary{positioning, angles}
\usetikzlibrary{arrows}
\usetikzlibrary{decorations.pathmorphing, decorations.pathreplacing}
\usetikzlibrary{decorations.markings}
\usetikzlibrary{calc,shapes}
\usetikzlibrary{decorations.markings}
\usepackage{mathtools}
\usepackage{lipsum}
\usepackage{comment}
\usepackage{float}    
\usetikzlibrary{shapes.geometric}

\let\phi\varphi

\usepackage{amsmath, amsthm, amssymb, amsfonts, xfrac}
\newtheorem{prop}{\textsc{Proposition}}[section]
\newtheorem{cor}[prop]{\textsc{Corollary}}
\newtheorem{thm}[prop]{\textsc{Theorem}}
\newtheorem*{thm*}{\textsc{Theorem}}
\newtheorem*{prop*}{\textsc{Proposition}}

\theoremstyle{definition}
\newtheorem{defn}[prop]{\textsc{Definition}}
\newtheorem{ex}[prop]{\textsc{Example}}

\newtheorem{remark}[prop]{\textsc{Remark}}

\newtheorem{lemma}[prop]{\textsc{Lemma}}

\newtheorem{question}{Open Question}

\usepackage{graphicx}
\newcommand{\Z}{\mathbb{Z}}
\newcommand{\AC}{\widetilde{\mathcal{AC}}}
\newcommand{\A}{\mathcal{A}^{\Z}}

\newcommand{\Q}{\mathbb{Q}}
\newcommand{\cc}{\mathcal{C}}
\newcommand{\Sy}{\mathcal{S}}
\newcommand{\W}{W}
\newcommand{\qb}{\mathfrak{qb}}

\newcommand{\R}{\mathbb{R}}
\newcommand{\F}{\mathbb{F}}
\newcommand{\G}{\mathcal{G}}
\newcommand{\C}{\widetilde{\mathcal{C}}}

\DeclareMathOperator{\Hom}{Hom}

\DeclareMathOperator{\rank}{rank}

\DeclareMathOperator{\lk}{lk}

\DeclareMathOperator{\genus}{genus}
\DeclareMathOperator{\im}{im}

\DeclareMathOperator{\sign}{sign}

\DeclareMathOperator{\SO}{SO}

\DeclareMathOperator{\Fix}{Fix}
\DeclareMathOperator{\tr}{tr}
\DeclareMathOperator{\id}{id}

\let\int\relax 
\DeclareMathOperator{\int}{int}

\usepackage{hyperref}

\title[Equivariant algebraic concordance of strongly invertible knots]{Equivariant algebraic concordance\\of strongly invertible knots}

\author{Alessio Di Prisa}
\address{Scuola Normale Superiore, Pisa, Italy \vskip.05in}
\email{\url{alessio.diprisa@sns.it}}
\urladdr{\url{https://sites.google.com/view/alessiodiprisa}}
\begin{document}
\maketitle

\begin{abstract}
By considering a particular type of invariant Seifert surfaces we define a homomorphism $\Phi$ from the (topological) equivariant concordance group of directed strongly invertible knots $\C$ to a new equivariant algebraic concordance group $\widetilde{\G}^\Z$.
We prove that $\Phi$ lifts both Miller and Powell's equivariant algebraic concordance homomorphism \cite{miller_powell} and Alfieri and Boyle's equivariant signature \cite{alfieri2021strongly}.
Moreover, we provide a partial result on the isomorphism type of $\widetilde{\G}^\Z$ and obtain a new obstruction to equivariant sliceness, which can be viewed as an equivariant Fox-Milnor condition. We define new equivariant signatures and using these we obtain novel lower bounds on the equivariant slice genus.
Finally, we show that $\Phi$ can obstruct equivariant sliceness for knots with Alexander polynomial one.
\end{abstract}

\section{Introduction}

A knot $K\subset S^3$ is said to be \emph{invertible} if there is an orientation-preserving homeomorphism $\rho$ of $S^3$ such that $\rho(K)=K$ and $\rho$ reverses the orientation on $K$. If such a homeomorphism can be taken to be a locally linear involution, we say that $K$ is \emph{strongly invertible}. Kawauchi \cite[Lemma 1]{kawauchi1979invertibility}proved that the two notions agree for hyperbolic knots, however, there exist examples of invertible knots which are not strongly invertible, see \cite[Section 5]{hartley1980knots}.

Since such an involution for a strongly invertible knot has always a non-empty fixed-point set, we know that it is always conjugate to an element of $\SO(4)$ by the solution of the Smith conjecture \cite{smith}.
As a consequence, we can think of a strongly invertible knot as a knot that is invariant under a $\pi$-rotation around some unknotted axis in $S^3$. The knot intersects the axis in two points, by which the axis is separated into two half-axes.

In \cite{sakuma} Sakuma defined a notion of \emph{direction} on a strongly invertible knot, which consists of an orientation of the axis of the involution together with the choice of one of the half-axes. Using this additional structure he was able to define unambiguously an operation of \emph{equivariant connected sum} between \emph{directed strongly invertible knots}.

Since strongly invertible knots are naturally equipped with an involution, it is natural to ask whether a strongly invertible slice knot is also \emph{equivariantly slice}, i.e.~it bounds a locally flat slice disk in $B^4$ which is invariant under a locally linear extension of the involution.
Similarly to the classical case, this leads to the definition of the \emph{equivariant concordance group} $\C$ as the set of classes of directed strongly invertible knots up to an appropriate definition of \emph{equivariant concordance}.
Recently in \cite{diprisa} we proved that $\C$ is not abelian\footnote{The main result \cite[Theorem 1.1.]{diprisa} of the paper relies on Donaldson's diagonalization theorem, which is inherently a smooth result. However, in \cite[Section 3.1]{diprisa} we give an alternate proof that $\C$ is not abelian, relying on twisted Alexander polynomials, which can be adapted in the topological category.}, which is in stark contrast with the classical concordance group.

Several authors \cite{sakuma,boyle2021equivariant,alfieri2021strongly,dai_mallick_stoffregen,miller_powell,diprisa_framba} have found invariants and obstructions for the equivariant concordance of strongly invertible knots.
In particular, in \cite{dai_mallick_stoffregen} the authors define several invariants for \emph{smooth} equivariant concordance using knot Floer homology. Using the lower bounds on the \emph{smooth equivariant slice genus} $\widetilde{g}_4^s$ provided by these invariants, they construct the first examples of strongly invertible knots with $\widetilde{g}_4^s(K)-g_4^s(K)$ arbitrarily large, where $g_4^s(K)$ is the classical smooth slice genus, answering \cite[Question 1.1]{boyle2021equivariant}.

Miller and Powell, in \cite{miller_powell} introduce a notion of \emph{equivariant algebraic concordance}, by studying the action of the strong inversion on the Blanchfield pairing on the Alexander module of a strongly invertible knot.
In this way, they define a homomorphism
$$\Psi:\C\longrightarrow\AC$$
from the equivariant concordance group to an \emph{equivariant algebraic concordance group} $\AC$ of \emph{equivariant Blanchfield pairings}. From the equivariant Blanchfield pairings they obtain new lower bounds on the equivariant slice genus, and they provide examples of genus one slice knots with arbitrarily large \emph{topological} equivariant slice genus $\widetilde{g}_4$.

In \cite{boyle2021equivariant,alfieri2021strongly} the authors define an equivariant version of the classical knot signature for directed strongly invertible knots, obtaining a group homomorphism $$\widetilde{\sigma}:\C\longrightarrow\Z.$$

In this paper we define a notion of \emph{equivariant algebraic concordance} for directed strongly invertible knots, analogous to Levine's algebraic concordance \cite{Levine1969a,Levine1969b}, by considering a particular \emph{type} of \emph{invariant Seifert surfaces}. In Theorem \ref{en_alg_conc} we construct a homomorphism $\Phi$ from the equivariant concordance group $\C$ to an \emph{equivariant algebraic concordance group} $\widetilde{\G}^\Z$ of \emph{equivariant Seifert systems}.

The homomorphism $\Phi$ is able to detect valuable information about equivariant concordance. For instance, it can be used to obstruct equivariant sliceness of knots with Alexander polynomial one, see Example \ref{ex:alexander_one}.
However, it seems a hard task to give a complete description of the group structure of $\widetilde{\G}^\Z$.
In Section \ref{sect:blanchfield} we introduce a more manageable quotient $\widetilde{\G}^\Z_r$ of this group that we call \emph{reduced equivariant concordance group} and we denote by $\Phi_r: \C\longrightarrow \widetilde{\G}^\Z_r$ the homomorphism obtained by composition.

\pagebreak
The main result of this paper is the following theorem, which sums up the results of Theorems \ref{g_to_ac} and \ref{equivalent_definition}.
\begin{thm}
The homomorphism $\Psi$ and the equivariant signature $\widetilde{\sigma}$ factor through $\widetilde{\G}^\Z_r$, i.e. they fit in the following commutative diagram.
\begin{center}
    \begin{tikzcd}
    &\C\ar[d,"\Phi_r"]\ar[rd,"\Psi"]\ar[ld,swap,"\widetilde{\sigma}"]&\\
    \Z&\widetilde{\G}^\Z_r\ar[l]\ar[r]&\AC
    \end{tikzcd}
\end{center}
\end{thm}
Moreover, in Theorem \ref{thm:grp_str} we obtain a partial result on the structure of $\widetilde{\G}^\Z$, determining the isomorphism type of $\widetilde{\G}_r^\Q$ which is the defined similarly to $\widetilde{\G}_r^\Z$, by using \emph{rational} equivariant Seifert forms. In particular, we have a natural inclusion $\widetilde{\G}_r^\Z\hookrightarrow\widetilde{\G}_r^\Q$ (see Lemma \ref{lemma:inclusion}) and the following theorem holds.
\begin{thm*}[\textbf{\ref{thm:grp_str}}]
The isomorphism type of the reduced rational equivariant algebraic concordance group is given by
$$
\widetilde{\G}_r^\Q\cong\Z^\infty\oplus\Z/2\Z^\infty\oplus\Z/4\Z^\infty\oplus\Z/8\Z^\infty.
$$
\end{thm*}

The same arguments used in the proof of Theorem \ref{thm:grp_str} fail to be easily adapted to study the (unreduced) group $\widetilde{\G}^\Z$. Therefore, we propose the following open problem that we would like to address in the future.

\begin{question}\label{question:grp_str}
Determine the full isomorphism type of the equivariant algebraic concordance group $\widetilde{\G}^\Z$.
\end{question}

As pointed out in Remark \ref{rem:non_surj} the composite map $\C\longrightarrow\widetilde{\G}_r^\Q$ is not surjective, since $\widetilde{\G}^Z_r$ is a proper subgroup. Hence we ask the following question.

\begin{question}
Determine the image of $\C$ in $\widetilde{\G}_r^\Q$. In particular, does there exist a directed strongly invertible knot $K$ whose image has order $8$ in $\widetilde{\G}_r^\Q$?
\end{question}

As a consequence of the investigation on the group structure of $\widetilde{\G}_r^\Q$, in Theorem \ref{thm:fox_milnor} we get the following obstruction to equivariant sliceness, which can be seen as an \emph{equivariant Fox-Milnor condition}.

\begin{thm*}[\textbf{\ref{thm:fox_milnor}}]
Let $K$ be a strongly invertible knot and let $\Delta_K(t)$ be its Alexander polynomial, normalized so that $\Delta_K(t)=\Delta_K(t^{-1})$ and $\Delta_K(1)=1$. If $K$ is equivariantly slice then $\Delta_K(t)$ is a square.
\end{thm*}

Theorem \ref{thm:fox_milnor} is especially fascinating due to its intriguing resemblance to a result of Hartley and Kawauchi \cite{hartley1979polynomials} which states that if a knot $K$ is \emph{strongly positive amphichiral} then $\Delta_K(t)$ is a square. While at the moment we are not able to formulate a precise conjecture, it would be interesting to understand better the relation between being equivariantly slice and strongly positive amphichiral for a strongly invertible knot.

The results of Section \ref{sect:grp_str} suggest naturally the definition of new \emph{equivariant signatures} $\{\widetilde{\sigma}_\lambda\}_{\lambda\in\R}$ and \emph{equivariant signature jumps} $\{\widetilde{J}_\lambda\}_{\lambda\in\R}$, which are homomorphisms
$$
\widetilde{\sigma}_\lambda,\;\widetilde{J}_\lambda:\C\longrightarrow\Z
$$
that we introduce in Section \ref{sect:new_equiv_sign}.
In Proposition \ref{prop:levine_tristram} we clarify the relation between $\widetilde{\sigma}_\lambda$, the Levine-Tristram signatures $\sigma_\omega$ proving the following.
\begin{prop*}[\textbf{\ref{prop:levine_tristram}}]
Let $K$ be a directed strongly invertible knot. Then for any $\lambda\leq 0$ and $\omega\in S^1$ such that $\lambda(\omega-1)^2=(\omega+1)^2$ we have
$$
\widetilde{\sigma}_\lambda(K)=\sigma_\omega(K),
$$
\end{prop*}
While the equivariant signatures coincide with the Levine-Tristram signatures for $\lambda\leq 0$, for positive values of $\lambda$ we actually get a new invariant of equivariant concordance (see Remark \ref{rem:lin_indep}).
The main result of Section \ref{sect:new_equiv_sign} is Theorem \ref{thm:genus_bound_jump}, which gives a new lower bound on the equivariant slice genus $\widetilde{g}_4$ of a strongly invertible knot.

\begin{thm*}[\textbf{\ref{thm:genus_bound_jump}}]
Given a directed strongly invertible knot $K$, for every $\lambda>0$, $\lambda\neq 1$ we have
$$\widetilde{g}_4(K)\geq\frac{|\widetilde{J}_\lambda(K)|}{4}.$$
\end{thm*}

Theorem \ref{thm:genus_bound_jump} can be used to obtain new examples of strongly invertible knots with $\widetilde{g}_4(K)-g_4(K)$ arbitrarily large (see Remark \ref{remark:big_equiv_genus}), where $g_4(K)$ is the classical (topological) slice genus.

\subsection*{Organization of the paper}
In Section \ref{sect:preliminaries} we briefly recall some notions and results on equivariant concordance and on algebraic concordance, and we introduce the definition of \emph{$n$-butterfly link}, which is a generalization of the \emph{butterfly link} \cite{boyle2021equivariant}.
Section \ref{ext_transv} contains some results on the extension and transversality of equivariant maps, that are used in the next section.
In Section \ref{sect:equiv_alg_conc} we use Proposition \ref{alg_slice_link} to motivate the definition of the equivariant algebraic concordance group $\widetilde{\G}^\Z$.
In Section \ref{sect:equiv_GL} we define a homomorphism from of $\widetilde{\G}^\Z$ to an equivariant version $\widetilde{W}(\Q)$ of the Witt group of $\Q$, and we show that the equivariant signature \cite{alfieri2021strongly} factors through $\widetilde{W}(\Q)$.
Section \ref{sect:grp_str} is dedicated to studying the group structure of $\widetilde{\G}_r^\Q$, which is a simpler variant of $\widetilde{\G}^\Z$. Using these results, we introduce in Section \ref{sect:new_equiv_sign} a new \emph{equivariant signature} function $\widetilde{\sigma}_\lambda$ and we describe how it can be used to obtain lower bounds on the equivariant slice genus.

Finally, in the Appendix \ref{apx:2_bridge} we provide a table of examples of application of Theorem \ref{thm:genus_bound_jump} on a family of $2$-bridge knots with at most $12$ crossings. 

\subsection*{Conventions}
We work in the topological category (see \cite{miller_powell, friedl2019survey} for details) unless otherwise specified. More precisely, we will implicitly consider
\begin{itemize}
    \item maps between manifolds to be continuous,
    \item submanifolds to be (properly) locally flat embedded,
    \item group actions (specifically involutions) on manifolds to be locally linear.
\end{itemize}
Throughout the paper we will refer to and use some results appearing in \cite[Section 4]{boyle2021equivariant} and \cite{diprisa_framba}. While in this paper such results are stated in the smooth category, we want to remark that the proofs can be adapted to work in the topological category.

\section{Preliminaries}\label{sect:preliminaries}

\subsection{Directed strongly invertible knots}\label{SIK}
We recall the definition of \emph{directed strongly invertible knots} and \emph{equivariant concordance group} following \cite{sakuma,boyle2021equivariant}.

Let $(K,\rho)$ be a strongly invertible knot.
By the resolution of the Smith conjecture \cite{smith} we know that $\rho$ acts on $S^3$ as a rotation around the \emph{axis} $\Fix(\rho)$, which is an unknotted $S^1$. Since the restriction of $\rho$ on $K$ is orientation-reversing, the fixed axis intersects $K$ in two points, which separate $\Fix(\rho)$ in two so-called \emph{half-axes}.

\begin{defn}
A \emph{direction} $h$ on a strongly invertible knot $(K,\rho)$ is the choice of one of the half-axes $h$ and an orientation on $\Fix(\rho)$. We say that $(K,\rho,h)$ is a \emph{directed strongly invertible knot}.
\end{defn}

\begin{defn}
We say that two directed strongly invertible knots $(K_i,\rho_i,h_i)$, $i=0,1$ are \emph{equivariantly isotopic} if there exists an orientation-preserving homeomorphism $\varphi:S^3\longrightarrow S^3$ such that:
\begin{itemize}
    \item $\varphi(K_0)=K_1$,
    \item $\varphi\circ\rho_0=\rho_1\circ\varphi$,
    \item $\varphi(h_0)=h_1$, preserving the chosen orientations on $h_0$ and $h_1$.
\end{itemize}
\end{defn}

We will often omit to specify the choice of strong inversion and direction when it is not strictly necessary to specify them.

\begin{remark}
A direction on $(K,\rho)$ induces an ordering on $K\cap\Fix(\rho)$: we say that the \emph{first fixed point} of $K$ is the initial point of the chosen half-axis, while the final point is the \emph{second fixed point}.
\end{remark}

\begin{defn}
Let $K$ and $J$, be two directed strongly invertible knots.
Their \emph{equivariant connected sum} $K\widetilde{\#}J$ is the directed strongly invertible knot obtained by cutting $K$ at its second fixed point and $J$ at its first fixed point, gluing the two knots and axes equivariantly in a way that is compatible with the orientations on the axes, and choosing the half-axis of the sum to be the union of the half-axes of the two components, as depicted in Figure \ref{equivariant_sum}.
\end{defn}

\begin{figure}[ht]
\centering
\begin{tikzpicture}

\node[anchor=south west,inner sep=0] at (0,0){\includegraphics[scale=0.3]{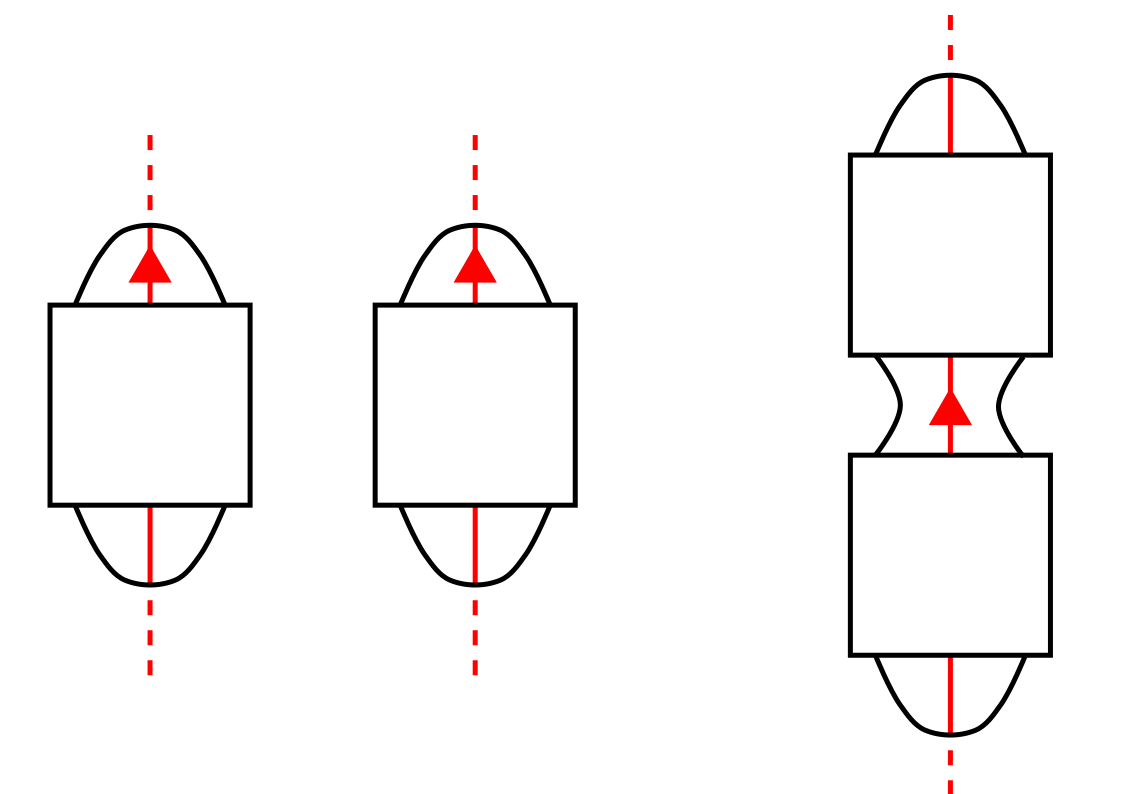}};
\node[label={$K$}] at (1.2,2.7){};
\node[label={$\widetilde{\#}$}] at (2.45,2.7){};

\node[label={$J$}] at (3.8,2.7){};
\node[label={$=$}] at (5.5,2.7){};
\node[label={$J$}] at (7.6,3.9){};
\node[label={$K$}] at (7.6,1.4){};

\end{tikzpicture}
    \caption{The equivariant connected sum of $K$ and $J$. The vertical axis (colored red) is the axis of the strong inversion. The chosen half-axis is the solid one.}
    \label{equivariant_sum}
\end{figure}
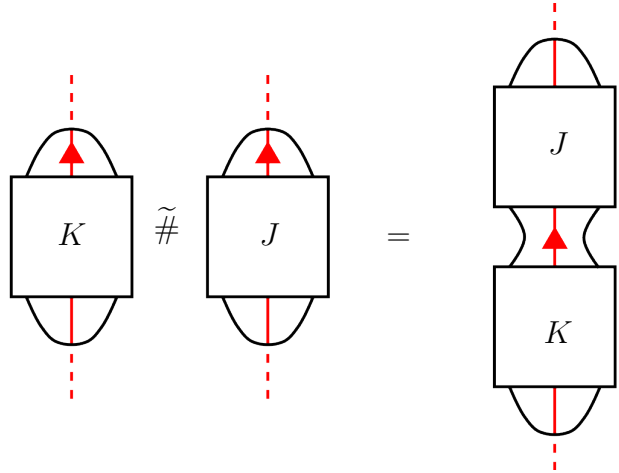

\begin{defn}
Let $(K,\rho,h)$ be a directed strongly invertible knot. We define
\begin{itemize}
    \item the \emph{mirror} of $(K,\rho,h)$ by $mK=(mK,\rho,h)$,
    \item the \emph{axis-inverse} of $(K,\rho,h)$ by $iK=(K,\rho,-h)$, where $-h$ is the direction given by the half-axis $h$ with the opposite orientation,
    \item the \emph{antipode} of $(K,\rho,h)$ by $aK=(K,\rho,h')$, where $h'$ is the direction given by the oriented half-axis complementary to $h$.
\end{itemize}
\end{defn}

\begin{defn}\label{def:equiv_slice}
Let $(K,\rho)$ be a strongly invertible knot. We say that $K$ is \emph{equivariantly slice} if there exists a locally flat slice disk $D\subset B^4$ for $K$, invariant with respect to a locally linear involution of $B^4$ extending $\rho$.
We define the \emph{equivariant slice genus} of $(K,\rho)$ as
$$\widetilde{g}_4(K)=\min_{\Sigma}\genus(\Sigma)$$
where $\Sigma$ ranges among the orientable locally flat surfaces in $B^4$ with boundary $K$, invariant under an involution extending $\rho$.
\end{defn}

\begin{defn}\label{def:equiv_conc}
We say that two directed strongly invertible knots $(K_i,\rho_i,h_i)$, $i=0,1$ are \emph{equivariantly concordant} if there exists a locally flat properly embedded annulus $C\cong S^1\times I\subset S^3\times I$, invariant with respect to some locally linear involution $\rho$ of $S^3\times I$ such that:
\begin{itemize}
    \item $\partial (S^3\times I,C)=(S^3,K_0)\sqcup -(S^3,K_1)$,
    \item $\rho$ is in an extension of the strong inversion $\rho_0\sqcup\rho_1$ on $S^3\times 0\sqcup S^3\times 1$,
    \item the orientations of $h_0$ and $-h_1$ induce the same orientation on the annulus (see \cite[Remark 2.12]{miller_powell}), $\Fix(\rho)$, and $h_0$ and $h_1$ are contained in the same component of $\Fix(\rho)\setminus C$.
\end{itemize}
\end{defn}

The operation of equivariant connected sum induces a group structure on the set $\C$ of classes of directed strongly invertible knots up to equivariant concordance. The class of the unknot gives the group identity, while the inverse of $K$ can be represented by $K^{-1}:=m(i(K))$.

Notice that, while the direction is essential to define an equivariant connected sum, $K$ is equivariantly slice if and only if $iK$ or $aK$ is so. Therefore we can consider the mirror, axis-inverse, and antipode as involutive maps from $\C$ to itself.

\subsection{Butterfly links}
In \cite{boyle2021equivariant} Boyle and Issa associate with a directed strongly invertible knot the so-called \emph{butterfly link}. Using this link they construct several invariants. In this section, we recall some of the invariants defined in \cite{boyle2021equivariant}.
Additionally, we introduce the definition of \emph{$n$-butterfly link} of a directly strongly invertible knot, which is important in the following sections. The $n$-butterfly link is a generalization of the butterfly link and it coincides with the definition in \cite{boyle2021equivariant} for $n=0$.

\begin{defn}\label{butterfly_link}
Let $(K,\rho,h)$ be a directed strongly invertible knot. Take an equivariant band $B$, parallel to the preferred half-axis $h$, which attaches to $K$ at the two fixed points.
Perform a band move on $K$ along $B$ such that the result is a $2$-component link. The linking number between the components of such a link depends on the number of twists of $B$ (see for example Figure \ref{bandmove}).
Observe that $\partial B\setminus K$ consists of two arcs parallel to $h$, which we orient as $h$. The arcs lie in different components of the link and we consider on each component the orientation induced from the respective arc.
The \emph{$n$-butterfly link} $L_b^n(K)$, is the $2$-components $2$-periodic link (i.e. the involution $\rho$ exchanges its components) obtained from such a band move on $K$ so that the linking number between its components is $n$.
\end{defn}

\begin{figure}
    \centering
    \includegraphics[scale=0.5]{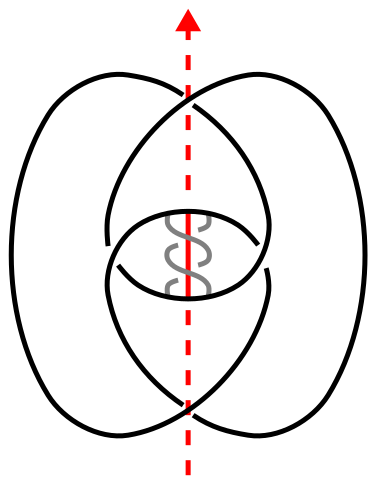}
    \caption{The band move (in grey) that produces the $0$-butterfly link of $4_1^+$.}
    \label{bandmove}
\end{figure}

Recall that a \emph{semi-orientation} on a link $L$ is the choice of an orientation on each component of $L$, up to reversing the orientation on all components simultaneously.

\begin{defn}
Define $\widehat{L}_b^n(K)$ to be the $n$-butterfly link of $K$ endowed with the opposite semi-orientation.
Observe that the semi-orientation on $\widehat{L}_b^n(K)$ makes the band move along $B$ coherent with the unique semi-orientation on $K$.
\end{defn}

With a slight abuse of notation, we will also call $\widehat{L}_b^n(K)$ the $n$-butterfly link of $K$.
\begin{remark}
Notice that the linking number between the components of $\widehat{L}_b^n(K)$, taken with respect to the chosen semi-orientation, is $-n$.
\end{remark}

\begin{defn}
Let $(L_i,\rho_i)$, $i=0,1$, be two $2$-component $2$-periodic links. We say that $(L_0,\rho_0)$ and $(L_1,\rho_1)$ are \emph{equivariantly concordant} if they bound two disjoint properly embedded locally flat annuli in $S^3\times I$, which are invariant with respect to some locally linear involution of $S^3\times I$ extending $\rho_0\sqcup\rho_1$.
\end{defn}

\begin{prop}\label{link_concordance}
Let $(K_i,\rho_i,h_i)$, $i=0,1$, be two equivariantly concordant directed strongly invertible knots. Then, $L_b^n(K_0)$ (resp. $\widehat{L}_b^n(K_0)$) is equivariantly concordant to $L_b^n(K_1)$ (resp. $\widehat{L}_b^n(K_1)$).
\end{prop}
\begin{proof}
The proof is identical to the proof of Proposition 2.6 in \cite{diprisa_framba}.
\end{proof}

\begin{remark}\label{slice_link}
It follows from the proposition above that if $K$ is equivariantly slice then also $\widehat{L}_b^0(K)$ is \emph{equivariantly slice}, i.e. it bounds two disjoint equivariant disks in $B^4$. On the other hand, since $\widehat{L}_b^0(K)$ is obtained by an equivariant band move from $K$ (which can be seen as a genus $0$ equivariant cobordism) we have that if $\widehat{L}_b(K)$ is equivariantly slice then so is $K$.
\end{remark}

Despite Remark \ref{slice_link}, it is not true in general that if $L_b^0(K)$ is equivariantly concordant to $L_b^0(J)$ then $K$ is equivariantly concordant to $J$.

\begin{defn}[\cite{boyle2021equivariant}]\label{butterfly_homomorphisms}
Let $K$ be a directed strongly invertible knot. Define
\begin{itemize}
    \item $\mathfrak{b}(K)$ to be the knot given by one component of $L_b^0(K)$,
    \item $\qb(K)$ to be the knot $L_b^0(K)/\rho$ in $S^3/\rho\cong S^3$.
\end{itemize}
As proven in \cite{boyle2021equivariant}, we have that $\mathfrak{b}$ induces a group homomorphism
$$
\mathfrak{b}:\C\longrightarrow\mathcal{C},
$$
where $\C$ is the classical (topological) knot concordance group.
\end{defn}

\begin{defn}\label{double}
Given an oriented knot $K$, its \emph{double} $\mathfrak{r}(K)$ is the directed strongly invertible knot given by $K\#r(K)$, with the involution $\rho$ that exchanges $K$ and $r(K)$ (the $\pi$-rotation around the vertical axis in Figure \ref{rK}). The direction on $\mathfrak{r}(K)$ is given as follows: the connected sum can be performed by a suitable band move along some band $B$, in grey in the figure, in such a way that $\Fix(\rho)\cap B$ is the half-axis $h$. We orient $h$ as the portion of $B$ lying on $K$ (note that $h$ is parallel to $B\cap K$).
\end{defn}

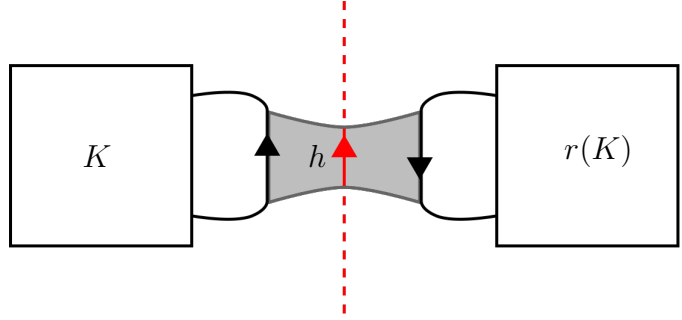
\begin{figure}[ht]
\centering
\begin{tikzpicture}

\node[anchor=south west,inner sep=0] at (0,0){\includegraphics[scale=0.3]{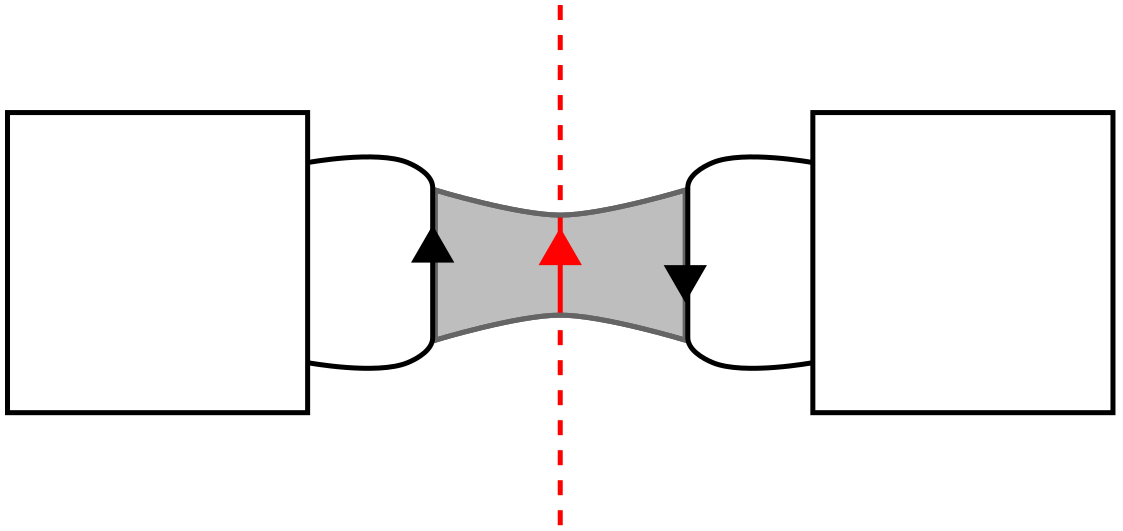}};
\node[label={$K$}] at (1.2,1.7){};
\node[label={$h$}] at (4.1,1.7){};

\node[label={$r(K)$}] at (7.8,1.7){};

\end{tikzpicture}
    \caption{The directed strongly invertible knot $\mathfrak{r}(K)$. The chosen half-axis is the solid one.}
    \label{rK}
\end{figure}

As proven by Boyle and Issa \cite{boyle2021equivariant}), $\mathfrak{r}$ defines a homomorphism
$$\mathfrak{r}:\mathcal{C}\longrightarrow\C.$$

It is immediate from the definitions that given an oriented knot $K$, the $0$-butterfly link of $\mathfrak{r}(K)$ is given by two split copies of $K$ (see again \cite{boyle2021equivariant} for details).
Therefore the composition $\mathfrak{b}\circ\mathfrak{r}:\mathcal{C}\longrightarrow\mathcal{C}$ is the identity homomorphism.

Hence, we get that $\mathfrak{b}$ is surjective, $\mathfrak{r}$ is injective, and that $\mathfrak{r}(\cc)$ is a copy of the classical concordance group, contained in the center of $\C$ (as noted in \cite{sakuma}).
As a consequence, we observe the following corollary.

\begin{cor}\label{split}
The equivariant concordance group splits as
$$\C=\ker(\mathfrak{b})\oplus\mathfrak{r}(\mathcal{C}).$$
\end{cor}

\subsection{Algebraic concordance}
In \cite{Levine1969a,Levine1969b} Levine defined a surjective homomorphism from the classical concordance group to a Witt group of Seifert forms, called the \emph{algebraic concordance group}, which is given by
\begin{gather*}
    \phi:\mathcal{C}\longrightarrow\G^{\Z}\\
    [K]\longmapsto [\theta_F]
\end{gather*}
where $F$ is a Seifert surface for $K$ and $\theta_F$ is the Seifert form on $H_1(F,\Z)$.

By taking the symmetrization of the Seifert form one obtains a group homomorphism
\begin{gather*}
\G^\Z\longrightarrow\W(\Q)\\
[A]\longmapsto[A+A^t]
\end{gather*}
where $\W(\Q)$ is the Witt group of non-degenerate symmetric forms on finite-dimensional $\Q$-vector spaces.
Denote by $\phi_W:\mathcal{C}\longrightarrow\W(\Q)$ the composition.
Clearly by composing $\phi_W$ with the signature homorphism $\sigma:\W(\Q)\longrightarrow\Z$ one get the knot signature $\sigma(K)$.

Given a (possibly nonorientable) spanning surface $F$ for a link $L$, Gordon and Litherland \cite{Gordon1978} defined a bilinear form
$$\G_F:H_1(F,\Z)\times H_1(F,\Z)\longrightarrow \Z$$
$$(a,b)\longmapsto \lk(\widetilde{a},b)$$
given by the linking number of $b$ with $a$ pushed off $F$ ``in both directions simultaneously''. This form is bilinear and symmetric and if $F$ is oriented it coincides with the symmetrization of the Seifert form.

In \cite{Gordon1978} Gordon and Litherland proved that it is possible to compute the signature of a knot from the Gordon-Litherland form of any spanning surface, by introducing a corrective term. We briefly recall some of the notation used in \cite{Gordon1978} and we observe in Proposition \ref{GL_refined} how the results of Gordon and Litherland allow us to compute not only the signature of a knot $K$ but the whole Witt class $\phi_W(K)$, using any spanning surface.
This fact is presumably known to the experts but we could not find it in the literature.

\begin{defn}\label{euler_number}
Let $F$ be a spanning surface for a knot $K$ and let $K^F$ be a longitude of $K$ which misses $F$.
The \emph{relative Euler number of $F$} is defined as $$e(F)=-\lk(K,K^F),$$
where $K$ and $K^F$ are coherently oriented. 
\end{defn}
Observe that since $K^F$ and $F$ are disjoint, $[K]=0\in H_1(S^3\setminus K^F,\Z/2\Z)$. Hence $e(F)$ is always an even integer.

\begin{defn}
Let $F_1$, $F_2$ be two surfaces in $S^3$ with $\partial F_1=\partial F_2$ and suppose that there exists a $3$-ball $B^3=B^1\times B^2\subset S^3\setminus \partial F_i$ such that
\begin{itemize}
    \item $F_1\cap B^3=\partial B^1\times B^2$,
    \item $F_2\cap B^3=B^2\times\partial B^1$,
    \item $F_1\setminus B^3=F_2\setminus B^3$.
\end{itemize}
In this situation, we say that $F_2$ is obtained from $F_1$ by a \emph{$1$-handle move}.
\end{defn}

\begin{figure}[ht]
\centering
\begin{tikzpicture}

\node[anchor=south west,inner sep=0] at (0,0){\includegraphics[scale=0.3]{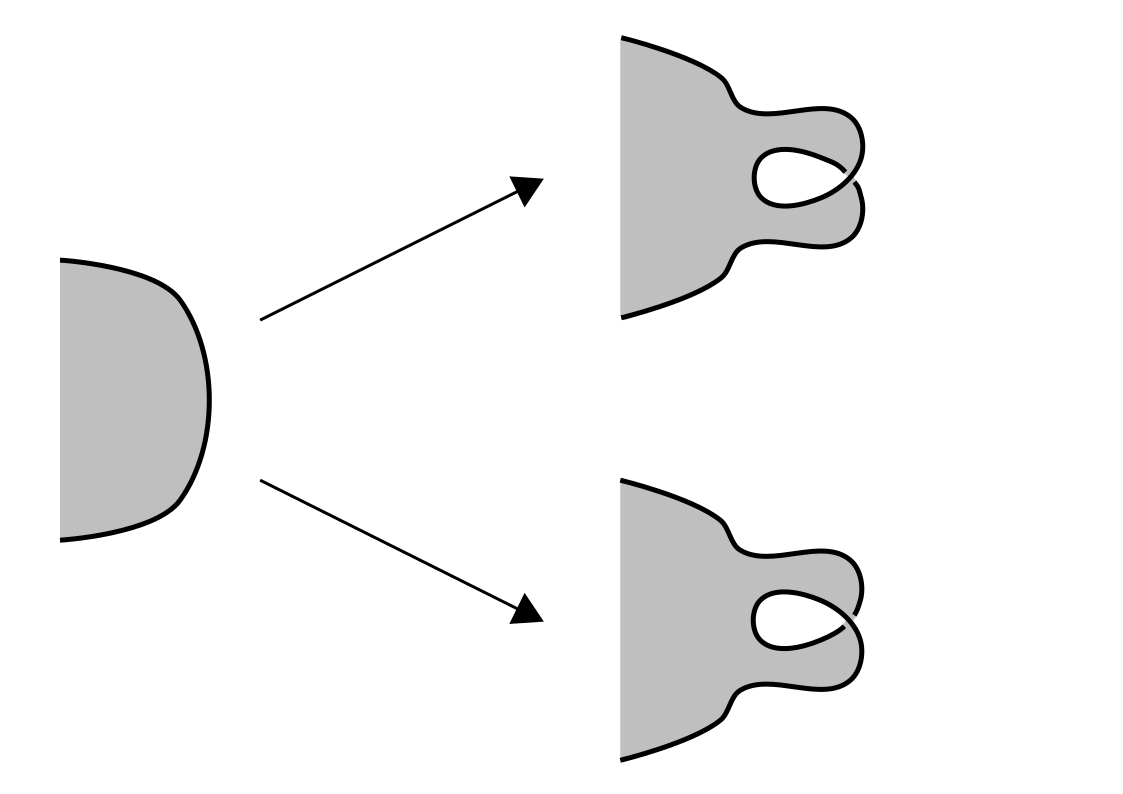}};
\node[label={$\varepsilon=+1$}] at (8,4.5){};
\node[label={$\varepsilon=-1$}] at (8,0.9){};
\node[label={$F$}] at (1,2.7){};
\end{tikzpicture}
    \caption{The addition of a half-twisted band.}
    \label{half_band}
\end{figure}
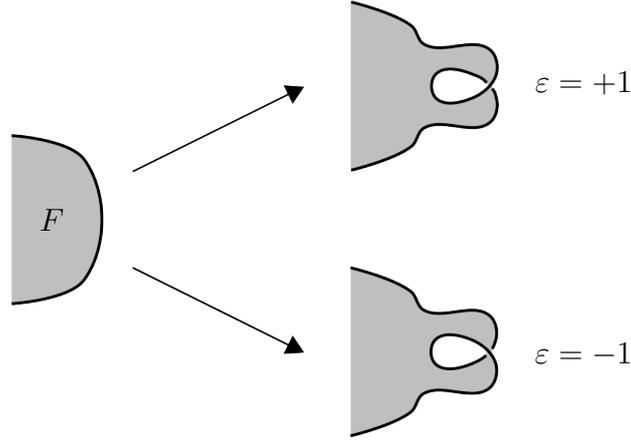

\begin{prop}\label{GL_refined}
Let $F$ be spanning surface for $K$, and $A$ a matrix representing Gordon-Litherland form $\G_F$ on $H_1(F)$.
Then, the Witt class of $K$ is represented by
$$
\begin{pmatrix}
A&0\\
0&\varepsilon Id
\end{pmatrix},
$$
where $\varepsilon=\sign(e(F))$ and the $Id$ block has size $n\times n$ with $n=|e(F)|/2$.
\end{prop}

\begin{proof}
Let $G$ be a Seifert surface for $K$ (in particular $e(G)=0$). By \cite[Theorem 11]{Gordon1978} we can obtain $G$ from $F$ by a finite sequence of the following moves (and their inverses):
\begin{itemize}
    \item ambient isotopy,
    \item $1$-handle moves,
    \item addition of a small half-twisted band at the boundary.
\end{itemize}

It is not difficult to check that the first two moves do not change the Witt class of the Gordon-Litherland form. By attaching a half-twisted band the Gordon-Litherland form and the relative Euler number change as
$$A\longrightarrow\begin{pmatrix}
A&0\\0&\varepsilon
\end{pmatrix},$$
$$e(F)\longmapsto e(F)-2\varepsilon,$$
where $\varepsilon=\pm1$ depends on the twist of the band, as in Figure \ref{half_band}.
Since the matrix $\begin{pmatrix}
1&0\\
0&-1
\end{pmatrix}$ is metabolic, if we attach two bands with opposite half-twists, the overall move leaves the Witt class unchanged.
The conclusion follows by observing that, up to algebraic cancellation, one has to attach $n=|e(F)|/2$ bands with the same half-twist.

\end{proof}

\section{Extension and transversality of equivariant maps}\label{ext_transv}
In this section, we show some results on the extension and transversality of equivariant maps. We use these results to prove Lemma \ref{3manifold}, which is fundamental for the constructions in Section \ref{sect:equiv_alg_conc}.

Let $X$ be a connected manifold with boundary, such that the inclusion of $\partial X$ in $X$ induces an isomorphism $H^1(X,\Z)\longrightarrow H^1(\partial X,\Z)$. Since $S^1$ is a $K(\Z,1)$, every map $\partial X\longrightarrow S^1$ can be extended to a map $X\longrightarrow S^1$, which is unique up to homotopy.

Consider now the $\Z/2\Z$-action on $S^1$ given by
$$\iota:S^1\longrightarrow S^1$$
$$z\longmapsto \overline{z},$$
and suppose that $X$ is endowed with a $\Z/2\Z$-action, generated by $\rho:X\longrightarrow X.$

\begin{lemma}\label{equiv_ext}
Let $(X,\partial X,\rho)$ be as above. Let $f:(\partial X,\rho)\longrightarrow (S^1,\iota)$ be an equivariant map, i.e. $f=\iota\circ f\circ\rho$, and suppose there exists $x_0\in \partial X$ such that $f(x_0)=1$. Then $f$ admits an equivariant extension $F:(X,\rho)\longrightarrow (S^1,\iota)$.
\end{lemma}

\begin{proof}
Let $G:X\longrightarrow S^1$ be a (possibly non equivariant) extension of $f$, and define $H:X\longrightarrow S^1$ as $H=G\cdot (\iota\circ G\circ \rho)$, where $\cdot$ is the group operation on $S^1$. By construction $H$ is equivariant.
Since $f$ is equivariant, we have that $\iota\circ G\circ \rho$ is another extension of $f$, and hence that it is homotopic to $G$. Therefore the induced maps are the same.
$$G_*=(\iota\circ G\circ\rho)_*:H_1(X,\Z)\longrightarrow H_1(S^1,\Z)=\Z.$$
It follows that $H_*=G_*+(\iota\circ G\circ\rho)_*=2G_*$ and then $H$ can be lifted to the two-fold covering.

\begin{center}
    \begin{tikzcd}
    &S^1\ar[d]&z\ar[d]\\
    X\ar[r,"H"]\ar[ru,"F", dashed]&S^1&z^2
    \end{tikzcd}
\end{center}

Choose the lift $F$ such that $F(x_0)=1$. Observe that since $f$ is equivariant, $H_{|\partial X}=f\cdot f$, then $F_{|\partial X}=f$. Therefore we only need to prove that $F$ is equivariant. 
Notice that $F\circ\rho$ is a lift of $H\circ\rho=\iota\circ H$. Therefore we have that $(F\cdot F\cdot (F\circ\rho)\cdot (F\circ\rho))=H\cdot \iota H\equiv 1$ and hence $F\cdot (F\circ\rho)\equiv \pm1$.
Since $F(x_0)=f(x_0)=1$ and $f$ is equivariant, we have that $F(x_0)\cdot F(\rho(x_0))=1$ and since $X$ is connected $F\cdot (F\circ\rho)\equiv 1$, i.e.
$F=\iota\circ F\circ \rho$.
\end{proof}

\begin{lemma}\label{3manifold}
Let $\rho:B^4\longrightarrow B^4$ be an orientation preserving, locally linear involution, with fixed-point set homeomorphic to a $2$-disk $D$.
Let $F\subset S^3$ and $\Sigma\subset B^4$ be two oriented locally flat surfaces with $L=\partial F=\partial\Sigma$.
Suppose that both $F$ and $\Sigma$ are $\rho$-invariant and that $\rho$ reverses their orientations.
If both $F$ and $\Sigma$ are disjoint from $\Fix(\rho)$, then there exists a $\rho$-invariant, oriented, compact, locally flat $3$-manifold $M\subset B^4$, disjoint from $\Fix(\rho)$ and such that $\partial M=F\cup\Sigma$.
\end{lemma}

\begin{proof}
Let $N(D)$ and $N(\Sigma)$ be equivariant, closed tubular neighbourhoods of $D$ and $\Sigma$ in $B^4$.
Let $X=B^4\setminus(\int N(D)\cup\int N(\Sigma))$.

Let $Y$ be the complement in $S^3$ of an equivariant tubular neighbourhood of $\partial F\cup\partial D$.
Let $N\cong F\times[-1,1]$ be and equivariant tubular neighbourhood of $F$ in $Y$. Observe that the restriction of $\rho$ acts on $N$ as
$$\rho:N\longrightarrow N$$
$$(x,t)\longmapsto (\rho(x),-t).$$

Let $\phi:\R\longrightarrow\R$ be a smooth, odd map such that $\phi'\geq0$ and
$$\phi(x)=\begin{cases}x\text{ for }|x|\leq1/2\\
\sign(x)\text{ for }|x|\geq2/3\end{cases}.$$
Then we can define an equivariant map
$$f:(N,\rho)\longrightarrow (S^1,\iota)$$
$$(x,t)\longmapsto e^{\pi i\phi(t)}$$
and we can extend it to $Y$ by setting $f$ to be $-1$ outside $N$.

Such $f$ is topologically transverse to $1\in S^1$ in the sense of \cite[Definition 10.7]{friedl2019survey} and $f^{-1}(1)$ is given by the union of $F$ and a nearby copy of $\Sigma$. Using Lemma \ref{equiv_ext} we can extend $f$ to an equivariant map
$$f:(X,\rho)\longrightarrow (S^1,\iota),$$
which in turn gives us the equivariant map
\begin{gather*}
\id_X\times f:(X,\rho)\longrightarrow(X\times S^1,\rho\times\iota)\\
x\longmapsto(x,f(x)).
\end{gather*}
Consider now the quotient spaces of $X$ and $X\times S^1$ by the respective involutions. We have the following commutative diagram
\begin{center}
\begin{tikzcd}
    X\times S^1\ar[d,"q"]\ar[r]&X\ar[d,"p"]\\
    (X\times S^1)/(\rho\times\iota)\ar[r]& X/\rho
\end{tikzcd}
\end{center}
where the vertical maps are $2$-fold regular covering maps,
Observe that the map $\id_X\times f$ induces a map between the quotients
$$
\overline{f}:X/\rho\longrightarrow (X\times S^1)/(\rho\times\iota).
$$
By construction $\overline{f}_{|\partial(X/\rho)}$ is topologically transverse to $(\partial X\times 1)/(\rho\times\iota)$.
According to \cite[Theorems 10.3 and 10.8]{friedl2019survey}, $\overline{f}$ is homotopic relative to the boundary to a map $\overline{g}$ transverse to $(X\times 1)/(\rho\times\iota)$.
Since $\overline{g}$ is homotopic to $\overline{f}$, we can lift $\overline{g}$ to an equivariant map
$$
g:(X,\rho)\longrightarrow(X\times S^1,\rho\times\iota)
$$
with $g_{|\partial X}=\id_{\partial X}\times f$ and $g$ topologically transverse to $X\times 1$.
Finally, $M=g^{-1}(X\times 1)$ is an equivariant, compact, orientable, locally flat $3$-dimensional submanifold of $X$ with $\partial M=F\cup\Sigma$.
\end{proof}

\section{Equivariant algebraic concordance}\label{sect:equiv_alg_conc}
In this section, we define an equivariant algebraic concordance group $\widetilde{\G}^\Z$ and a homomorphism $\Phi:\C\longrightarrow\widetilde{\G}^\Z$.
We compare $\widetilde{\G}^\Z$ with the equivariant algebraic concordance group $\AC$ defined in \cite{miller_powell}.
Finally, we use $\widetilde{\G}^\Z$ to obtain a lower bound on the equivariant slice genus of a strongly invertible knot.

\subsection{Equivariant Seifert systems}

\begin{defn}\label{inv_seifert_surface}
Let $(K,\rho,h)$ be a directed strongly invertible knot. An \emph{invariant Seifert surface of type $n$} for $K$ is a connected, orientable surface $F\subset S^3$ such that:
\begin{itemize}
    \item $F$ is $\rho$-invariant i.e. $\rho(F)=F$,
    \item $h=\Fix(\rho)\cap F$,
    \item the surface $\widehat{F}$ obtained from $F$ by equivariantly cutting along $h$ is a $\rho$-invariant Seifert surface for $\widehat{L}_b^n(K)$.
\end{itemize}
\end{defn}

\begin{prop}\label{stab_surface}
For any directed strongly invertible knot $K$ and every $n\in\Z$ there exists an invariant Seifert surface of type $n$.
\end{prop}

\begin{proof}
From \cite{hirasawa_hiura_sakuma} we know that for any $(K,\rho,h)$ there exists a $\rho$-invariant Seifert surface $F$ such that $\Fix(\rho)\cap F=h$.
Cutting $F$ along $h$ we obtain a (possibly disconnected) orientable surface $\widetilde{F}$ and the linking number between the components of $\partial\widetilde{F}$ would not be generally $-n$.
Now let $G$ be the equivariant Seifert surface for the unknot described in Figure \ref{unknot_surface}.

\begin{figure}[ht]
\centering
\begin{tikzpicture}

\node[anchor=south west,inner sep=0] at (0,0){\includegraphics[scale=0.6]{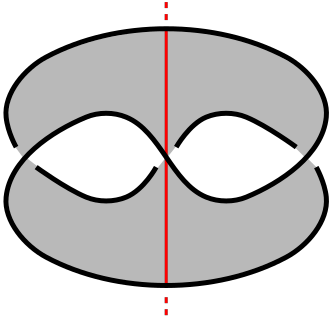}};

\end{tikzpicture}
    \caption{The invariant Seifert surface $G$ for the unknot.}
    \label{unknot_surface}
\end{figure}
Observe that by cutting $G$ along the fixed-point set we obtain a Seifert surface for a link with linking number $+1$ between its components.
In other words, $G$ is an invariant Seifert surface of type $-1$ for the unknot.
Therefore, by taking the equivariant connected sum of $F$ with an appropriate number of copies of $G$ and/or its mirror image, we easily get an invariant Seifert surface of type $n$ for $K$.

\end{proof}

As a consequence of Lemma \ref{3manifold}, we have the following proposition.

\begin{prop}\label{alg_slice_link}
Let $(K,\rho,h)$ be a directed strongly invertible knot and $F$ be an equivariant Seifert surface for $\widehat{L}_b^n(K)$. Suppose that $\widehat{L}_b^n(K)$ bounds an orientable surface $\Sigma\subset B^4$ invariant under an involution of $B^4$ extending $\rho$ (which we still denote by $\rho$). Assume that $\rho$ has no fixed point on $\Sigma\cup F$. Denote by $g_F$ and $g_{\Sigma}$ the genus of $F$ and $\Sigma$ respectively. Then there exists a $\rho_*$-invariant submodule $H\subset H_1(F,\Z)$ such that:
\begin{itemize}
    \item $\rank H\geq g_F-g_{\Sigma}$ if $\Sigma$ is connected and $\rank H\geq g_F-g_{\Sigma}+1$ if $\Sigma$ is not connected,
    \item the Seifert form of $F$ vanishes on $H$,
    \item for every $\alpha\in H$, the linking number between $\alpha$ and the fixed axis is zero.
\end{itemize}
\end{prop}

\begin{proof}
By Lemma \ref{3manifold} there exists a $\rho$-invariant oriented $3$-manifold $M\subset B^4$, such that $\partial M=F\cup\Sigma$ and $M\cap\Fix(\rho)=\emptyset$.

Denote by $V$ the kernel of
$H_1(\partial M,\Q)\longrightarrow H_1(M,\Q).$ It is easy to see that $2\cdot\dim V=\dim H_1(\partial M,\Q)=\genus(\partial M)$, by standard duality argument (half-lives, half-dies principle) or by computing the Euler characteristic of the exact sequence of the couple $(M,\partial M)$.

Suppose now that $\Sigma$ is connected. Then it is easy to see that $\genus(\partial M)=g_F+g_{\Sigma}+1$ and that the map induced by the inclusion $i_*:H_1(F,\Q)\longrightarrow H_1(\partial M,\Q)$ is injective.
Since $\dim V=g_F+g_{\Sigma}+1$ and $\dim H_1(F,\Q)=2g_F+1$, we have that the preimage $W$ of $V$ in $H_1(F,\Q)$ has dimension at least $g_F-g_{\Sigma}$.

Suppose now that $\Sigma$ is not connected. Then $\genus(\partial M)=g_F+g_{\Sigma}$ and the map induced by the inclusion $i_*:H_1(F,\Q)\longrightarrow H_1(\partial M,\Q)$ has kernel of dimension $1$.
Since $\dim V=g_F+g_{\Sigma}$ and $\dim H_1(F,\Q)=g_F+1$, we have that the preimage $W$ of $V$ in $H_1(F,\Q)$ has dimension at least $1+g_F-g_{\Sigma}$.

Define $H\subset H_1(F,\Z)$ to be
$$H=\{x\in H_1(F,\Z)\;|\;\exists n\in\Z, n\neq 0, nx\in W\}.$$

It is a well-known fact that the Seifert form of $F$ is identically zero on $H$. Since all of the maps considered are equivariant, we get that also $H$ is invariant under the action of $\rho_*$ on $H_1(F,\Z)$.

By the considerations above, the rank of $H$ satisfies the inequalities stated in the proposition.

Finally, let $\alpha\in H$ and let $\Delta\subset M$ be a $2$-chain such that $\partial \Delta=n\alpha$ for some integer $n\neq0$. Since $M$ is disjoint from the disk $D$ of fixed points, it follows that $\lk(n\alpha, \partial D)=\#(\Delta\cap D)=0$, hence $\lk(\alpha,\partial D)=0$.
\end{proof}

We use now the result given Proposition \ref{alg_slice_link} to define a notion of equivariant algebraic concordance for directed strongly invertible knots.

\begin{defn}
Let $R$ be a commutative and unital ring. An \emph{equivariant Seifert system} is a tuple $(\theta,\rho,h,\widetilde{\lk})$, where

\begin{itemize}
\item $\theta:M\times M\longrightarrow \Z$ is a bilinear form on a free $\Z$-module $M$ of even rank,
\item $\rho:M\longrightarrow M$ is a linear involution,
\item $\theta(\rho(x),\rho(y))=\theta^t(x,y):=\theta(y,x)$ for every $x,y\in M$,
\item $\theta-\theta^t$ is unimodular,
\item $h,\widetilde{\lk}\in\Hom(M,\Z)$,
\item $h\circ\rho=-h$,
\item $\widetilde{\lk}\circ\rho=\widetilde{\lk}$.
\end{itemize}

An equivariant Seifert system $(\theta,\rho,h,\widetilde{\lk})$ on $M$ is said to be \emph{equivariantly metabolic} if there exists a submodule $H\subset M$ such that
\begin{itemize}
\item $\rank M=2\cdot\rank H$
    \item $\rho(H)=H$, i.e. $H$ is \emph{$\rho$-invariant},
    \item $\theta$ is identically zero on $H\times H$,
    \item $H\subset \ker(h)\cap\ker(\widetilde{\lk})$.
\end{itemize}
\end{defn}

Let $(K,\rho,h)$ be a directed strongly invertible knot and let $F$ be an invariant Seifert surface for $K$ of type $n$ for some $n$. Fix an auxiliary orientation on $F$.

We see now how $F$ determines an equivariant Seifert system.
Since $\rho$ reverses the orientation on $F$, it is immediate to check that $\theta_F(\rho_*(x),\rho_*(y))=\theta_F(y,x)$ for every $x,y\in H_1(F,\Z)$, where $\theta_F$ is the Seifert form of $F$. Since $h\subset F$, we have that $h$ represents a class in $H_1(F,\partial F,\Z)$. By duality and universal coefficients, we can consider $h$ as a homomorphism $h:H_1(F,\Z)\longrightarrow \Z$, which maps an oriented curve $c$ in $F$ to the algebraic intersection $\#(c\cap h)$.
Finally, let $A$ be the oriented fixed axis of $\rho$. Then we have a homomorphism
$$\widetilde{\lk}:H_1(F,\Z)\longrightarrow \Z$$
$$c\longmapsto \lk(c^+,A)+\lk(c^-,A),$$
where the $c^{\pm}$ is a nearby copy of $c$ outside $F$ in the positive/negative direction.
It is immediate to check that the tuple $(\theta_F,\rho_*,h,\widetilde{\lk})$ is an equivariant Seifert system.
We will denote by $\Sy(F)$ the equivariant Seifert system determined by $F$.

\begin{defn}\label{orth_sum}
Let $(\theta_i,\rho_i,h_i,\widetilde{\lk}_i)$ for $i=1,2$ be two equivariant Seifert systems defined over $M$ and $N$ respectively. Their \emph{orthogonal sum} $(\theta_1,\rho_1,h_1,\widetilde{\lk}_1)\oplus(\theta_2,\rho_2,h_2,\widetilde{\lk}_2)$ is the tuple $(\theta,\rho,h,\widetilde{\lk})$ defined by
$$\theta:(M\oplus N)\times(M\oplus N)\longrightarrow\Z$$
$$((x_1,x_2),(y_1,y_2))\longmapsto \theta_1(x_1,y_1)+\theta_2(x_2,y_2)$$
$$\rho:M\oplus N\longrightarrow M\oplus N$$
$$\rho(x,y)=(\rho_1(x),\rho_2(y))$$
$$h,\widetilde{\lk}:M\oplus N\longrightarrow\Z$$
$$h(x,y)=h_1(x)+h_2(y)$$
$$\widetilde{\lk}(x,y)=\widetilde{\lk}_1(x)+\widetilde{\lk}_2(y).$$
We say that $(\theta_i,\rho_i,h_i,\widetilde{\lk}_i)$, $i=1,2$ are \emph{equivariantly concordant} if
the orthogonal sum between $(\theta_1,\rho_1,h_1,\widetilde{\lk}_1)$ and $(-\theta_2^t,\rho_2,h_2,\widetilde{\lk}_2)$ is equivariantly metabolic.
\end{defn}

\begin{defn}\label{equiv_alg_conc}
We define the \emph{equivariant algebraic concordance group} $\widetilde{\G}^{\Z}$ to be the set of equivalence classes of equivariant Seifert systems up to equivariant concordance. It is not difficult to prove that the operation of orthogonal sum defines a group structure on $\widetilde{\G}^{\Z}$, by adapting the proof of Levine \cite{Levine1969a,Levine1969b} in the case of the classical algebraic concordance.
\end{defn}

\begin{thm}\label{en_alg_conc}
Let $F\subset S^3$ be an invariant Seifert surface of type $0$ for a directed strongly invertible knot $(K,\rho,h)$ and choose an orientation on $F$. The class of the equivariant Seifert system $\Sy(F)$ in $\widetilde{\G}^{\Z}$ depends only on the equivariant concordance class of $K$.
In particular, we have a well-defined group homomorphism
$$\Phi:\C\longrightarrow\widetilde{\G}^{\Z}$$
$$[K,\rho,h]\longmapsto[\Sy(F)].$$
\end{thm}

\begin{proof}
Let $G$ be an invariant surface of type $0$ for another directed strongly invertible knot $J$. We can equivariantly perform the connected sum of $F$ and $G$ along their boundary so that $F\natural G$ is an invariant Seifert surface of type $0$ for the $K\widetilde{\#}J$. It is immediate to see that $\Sy(F\natural G)=\Sy(F)\oplus\Sy(G)$.
Therefore, to prove that the homomorphism $\Phi$ is well defined it is sufficient to show that $\Sy(F)$ is equivariantly metabolic whenever the knot $K=\partial F$ is equivariantly slice.

Let $\widehat{F}$ be the equivariant Seifert surface for $\widehat{L}_b^0(K)$ obtained by cutting $F$.
By Proposition \ref{alg_slice_link} there exists a $\rho_*$-invariant submodule $H$ of $H_1(\widehat{F},\Z)$, such that $2\rank H=\rank H_1(\widehat{F},\Z)+1$ and the Seifert form of $\widehat{F}$ vanishes on it.

Observe that we can regard $H_1(\widehat{F},\Z)$ as a $\rho_*$-invariant codimension $1$ submodule in $H_1(F,\Z)$ through the map induced by the inclusion. Moreover, $H_1(\widehat{F},\Z)$ is easily identified with the kernel of $h:H_1(F,\Z)\longrightarrow\Z$.
The restriction of the Seifert form of $F$ on $H_1(\widehat{F},\Z)$ clearly coincides with the Seifert form of $\widehat{F}$.
Again by Proposition \ref{alg_slice_link}, the linking number homomorphism $\widetilde{\lk}:H_1(F,\Z)\longrightarrow\Z$ vanishes on $H$.
Therefore $H$ is an equivariant metabolizer for the equivariant Seifert system $\Sy(F)$.
\end{proof}

\begin{remark}
Let $G$ be the equivariant Seifert surface of type $-1$ for the unknot described in Figure \ref{unknot_surface}.
Let $F$ be an invariant Seifert surface of type $n$ for a directed strongly invertible knot $(K,\rho,h)$.
Then by the proof of Theorem \ref{en_alg_conc} and Proposition \ref{stab_surface} follows easily that we can compute the equivariant algebraic concordance class of $K$ by
$$\Phi(K)=[\Sy(F)+n\Sy(G)]\in\widetilde{\G}^\Z.$$

Similarly, observe that in order to compute $\Phi(K)$ it is not relevant in Definition \ref{inv_seifert_surface} that $F\setminus h$ is connected. In fact, suppose that $F\setminus h$ is not connected. Then $F\natural G\natural\overline{G}$ is an equivariant Seifert surface of type $0$ for $K$, where $\overline{G}$ is the mirror image of the surface $G$ in Figure \ref{unknot_surface}. Hence, by definition $\Phi(K)=[\mathcal{S}(F\natural G\natural\overline{G})]=[\mathcal{S}(F)]+[\mathcal{S}(G)]+[\mathcal{S}(\overline{G})].$ Since $[\mathcal{S}(\overline{G})]=-[\mathcal{S}(G)],$ it follows that $\Phi(K)=[\mathcal{S}(F)]$.
\end{remark}

\begin{prop}
Let $\mathcal{A}$ be the concordance group of algebraically slice knots, i.e. the kernel of $\phi:\cc\longrightarrow\mathcal{G}^\Z$.
Then the kernel of $\Phi:\C\longrightarrow\widetilde{\G}^\Z$ contains a copy of $\mathcal{A}$, namely $\mathfrak{r}(\mathcal{A})\subset\ker(\Phi)$.
\end{prop}

\begin{proof}
Let $K$ be an oriented knot representing a class in $\mathcal{A}$ and let $F$ be a Seifert surface for $K$. Then we can compute $\Phi(\mathfrak{r}(K))$ using as invariant surface $\mathfrak{r}(F)=F\natural r(F)$, where analogously to Definition \ref{double}, the involution of $\mathfrak{r}(K)$ exchange $F$ and $r(F)$.

Identifying $H_1(\mathfrak{r}(F),\Z)\cong H_1(F,\Z)\oplus H_1(F,\Z)$, it is not difficult to see that the equivariant Seifert system of $\mathfrak{r}(F)$ is of type $\Sy(\mathfrak{r}(F))=(\theta,\rho,0,0)$, where
$$\theta=\begin{pmatrix}
\theta_F&0\\
0&\theta_F^t
\end{pmatrix}$$
$$\rho=\begin{pmatrix}
0&\id\\
\id&0
\end{pmatrix}.$$
Therefore, if $H\subset H_1(F,\Z)$ is a metabolizer of $\theta_F$ then $H\oplus H\subset H_1(\mathfrak{r}(F),\Z)$ is an equivariant metabolizer for $\Sy(\mathfrak{r}(F))$.
Since $\mathfrak{r}$ is injective (see Corollary \ref{split}), we have that $\ker(\Phi)$ contains a copy of $\mathcal{A}$.
\end{proof}

\begin{remark}
As a consequence, it follows from \cite{jiang1981simple} that $\ker(\Phi)$ contains a subgroup isomorphic to $\Z^\infty$, and from \cite{livingston1999order} that it contains a subgroup isomorphic to $\Z_2^\infty$.
\end{remark}

\subsection{Equivariant Blanchfield pairing}\label{sect:blanchfield}
In this section we show that $\Phi:\C\longrightarrow \widetilde{\G}^\Z$ lifts the homomorphism $\Psi:\C\longrightarrow\widetilde{\mathcal{AC}}$ defined in \cite{miller_powell}. In particular $\Psi$ factors through a \emph{reduced} version $\widetilde{\G}_r^\Z$ of the equivariant algebraic concordance group.

\begin{defn}\label{reduced}
An \emph{equivariant Seifert form} over a ring $R$ is a couple $(\theta,\rho)$, where

\begin{itemize}
\item $\theta:M\times M\longrightarrow R$ is a bilinear form on a free $R$-module $M$ of even rank,
\item $\rho:M\longrightarrow M$ is a $R$-linear involution,
\item $\theta(\rho(x),\rho(y))=\theta^t(x,y):=\theta(y,x)$ for every $x,y\in M$,
\item $\theta-\theta^t$ is unimodular, i.e. induces an isomorphism between $M$ and $\Hom(M,R)$,
\end{itemize}

We say that an equivariant Seifert form $(\theta,\rho)$ on $M$ is \emph{equivariantly metabolic} if there exists a submodule $H\subset M$ such that
\begin{itemize}
\item $\rank M=2\cdot\rank H$
    \item $\rho(H)=H$, i.e. $H$ is \emph{$\rho$-invariant},
    \item $\theta$ is identically zero on $H\times H$.
\end{itemize}
\end{defn}

Similarly to Definition \ref{orth_sum} and \ref{equiv_alg_conc} we can define a notion of orthogonal sum between equivariant Seifert forms and construct the \emph{reduced equivariant algebraic concordance group} $\widetilde{\G}_r^{\R}$ as the set of equivalence classes of equivariant Seifert forms over $R$ up to equivariant concordance.

In the following we mainly focus on equivariant Seifert forms over $\Z$ and we will often omit to specify the ring $R$, implying $R=\Z$.
Only Section \ref{sect:grp_str} will be mostly devoted to studying equivariant Seifert forms over $\Q$.

Clearly, there exists a forgetful homomorphism
$$r:\widetilde{\G}^{\Z}\longrightarrow \widetilde{\G}_r^\Z,$$
which is surjective, since it admits a natural section
$$s:\widetilde{\G}_r^\Z\longrightarrow\widetilde{\G}^\Z$$
given by mapping an equivariant Seifert form $(\theta,\rho)$ to the equivariant Seifert system $(\theta,\rho,0,0)$.
In particular $\widetilde{\G}^\Z$ splits as
$$\widetilde{\G}^\Z\cong \widetilde{\G}_r^\Z\oplus\ker(r).$$

We will denote by $\Phi_r$ the map given by the composition $$\Phi_r:\C\xrightarrow{\Phi}\widetilde{\G}^\Z\xrightarrow{r}\widetilde{\G}_r^\Z.$$

Levine \cite{Levine1969a,Levine1969b} showed that the algebraic concordance group is isomorphic to $\Z^{\infty}\oplus\Z_2^{\infty}\oplus\Z_4^{\infty}$ and that the knot concordance group surjects onto it.
In Section \ref{sect:grp_str} we provide a partial result, similar to Levine's one, on the structure of $\widetilde{\G}^\Z_r$.

The results in Section \ref{sect:grp_str} are obtained by adapting the arguments used in \cite{Levine1969a, Levine1969b} to the strongly invertible setting. However, these ideas do not generalize easily to study $\widetilde{\G}^\Z$: while the restriction on the equivariant metabolizers given by the homomorphism $h$ and $\widetilde{\lk}$ provide valuable information (as shown in Example \ref{ex:alexander_one}), it is not clear how to adapt these arguments to manage these additional restrictions.

In \cite{miller_powell} Miller and Powell study the action of the strong inversion $\rho$ of a strongly invertible knot $K$ on its Alexander module $\A(K)$ and on the Blanchfield pairing on $\A(K)$.
In particular, they show that the action induced by $\rho$ on $\A(K)$ is an anti-isometry of the Blanchfield pairing (Proposition 2.8). Moreover, they define an \emph{equivariant algebraic concordance group} $\AC$ as the Witt group of \emph{abstract equivariant Blanchfield pairings} (Definition 4.3) and they prove that taking the Blanchfield form of a strongly invertible knot $(K,\rho)$ together with the involution on $\A(K)$ induced by $\rho$ defines a homomorphism (Proposition 4.6)
$$\Psi:\C\longrightarrow\AC.$$

In \cite{friedl2017calculation} the authors prove that the Alexander module and the Blanchfield pairing of $K$ can be expressed in terms of the Seifert form of a Seifert surface $F$. If $A$ is a matrix representing the Seifert form of $F$, with respect to a basis $\mathcal{B}$, and $\genus(F)=g$ then $\A(K)\cong\Z[t^{\pm1}]^{2g}/(tA-A^t)\Z[t^{\pm1}]^{2g}$ and under this identification the Blanchfield pairing is equivalent to
$$\mathcal{BL}:\A(K)\times\A(K)\longrightarrow\Q(t)/\Z[t^{\pm1}]$$
$$(x,y)\longmapsto x(t-1)(A-tA^t)^{-1}\overline{y}$$
where $\overline{\cdot}$ is the $\Z$-linear involution given by $t\longmapsto t^{-1}$.
It is not difficult to see, as pointed out in the examples in \cite{miller_powell}, that if $F$ is $\rho$-invariant and $P$ is the matrix representing the action of $\rho$ on $H_1(F,\Z)$ with respect to the basis $\mathcal{B}$, we have that the action of $\rho$ on $\A(K)$ can be read as
$$\rho_*:\A(K)\longrightarrow\A(K)$$
$$x\longmapsto P\overline{x}.$$

The same construction carried out for abstract equivariant Seifert forms and abstract equivariant Blanchfield pairings proves the following theorem.
\pagebreak
\begin{thm}\label{g_to_ac}
There exists a natural group homomorphism
$$\widetilde{\G}_r^\Z\longrightarrow\AC$$
that makes the following diagram commutative
\begin{center}
    \begin{tikzcd}
    \C\ar[r,"\Phi"]\ar[rrd,swap,"\Psi"]&\widetilde{\G}^\Z\ar[r,"r"]&\widetilde{\G}_r^\Z\ar[d]\\
    &&\AC.
    \end{tikzcd}
\end{center}
\end{thm}

It follows from its definition that $\AC$ does not distinguish a directed strongly invertible knot from its antipode.
On the other hand, in Section \ref{sect:equiv_GL} we prove that the equivariant signature \cite{alfieri2021strongly} can be retrieved from $\widetilde{\G}_r^\Z$. Since the equivariant signature depends on the choice of half-axis for a strongly invertible knot, so is $\widetilde{\G}_r^\Z$ (see Remark \ref{equiv_signature_direction}).

We conclude with the following example, which shows that the equivariant algebraic concordance class of a knot is able to obstruct smooth as well as topological equivariant sliceness for knots with trivial Alexander polynomial, contrary to $\AC$.

\begin{ex}\label{ex:alexander_one}
Consider the knot $K13n1496$ as the directed strongly invertible knot $(K,\rho,h)$ that bounds the surface $F$ in Figure \ref{alexander_one}, where the strong inversion is given by the $\pi$-rotation around the vertical axis and the chosen oriented half-axis is the red one in the figure.

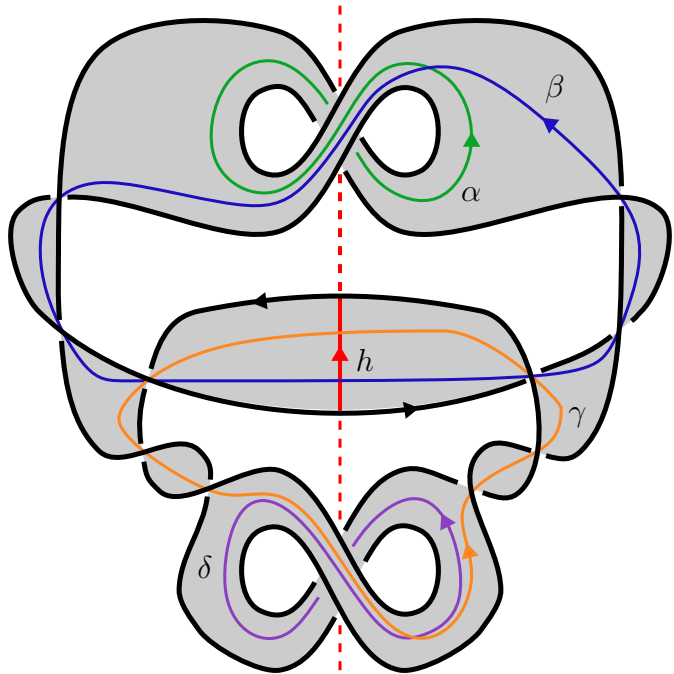
\begin{figure}[ht]
\centering
\begin{tikzpicture}

\node[anchor=south west,inner sep=0] at (0,0){\includegraphics[scale=0.5]{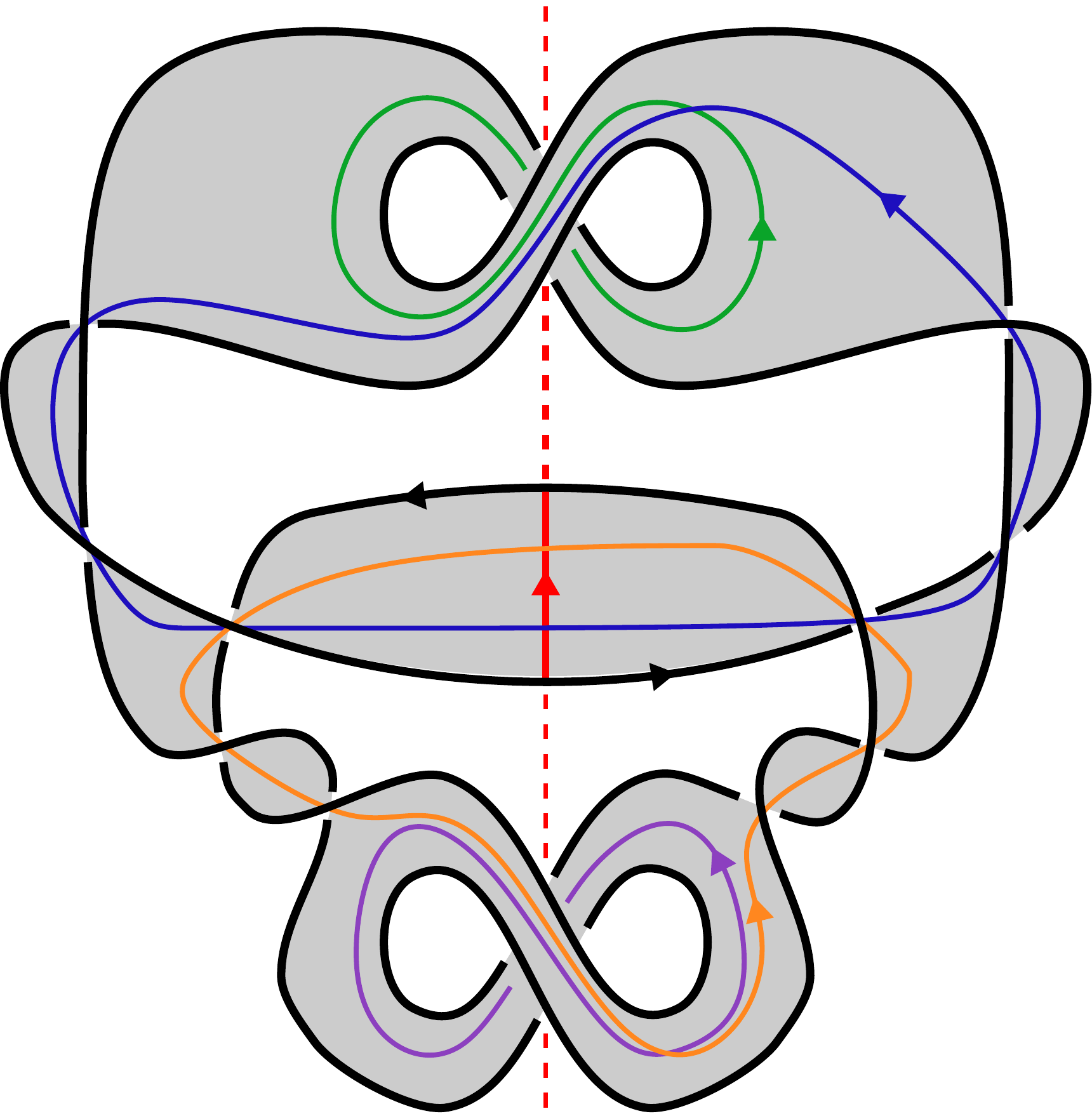}};
\node[label={$\alpha$}] at (6.1,6){};
\node[label={$\beta$}] at (7.2,7.3){};
\node[label={$\gamma$}] at (7.5,3){};
\node[label={$\delta$}] at (2.6,1){};
\node[label={$h$}] at (4.7,3.75){};
\end{tikzpicture}
    \caption{The invariant surface $F$ with boundary $K13n1496$. The chosen half-axis $h$ is the solid one. The curves $\alpha,\beta,\gamma,\delta$ form a basis of $H_1(F,\Z)$.}
    \label{alexander_one}
\end{figure}

One can easily check that $K$ has trivial Alexander polynomial and hence that its image is trivial in the equivariant algebraic concordance group $\AC$ defined in \cite{miller_powell}. However, Boyle and Issa \cite{boyle2021equivariant} prove that $K$ is not equivariantly slice. We show that the same result can be obtained by using $\widetilde{\G}^\Z$.

The surface $F$ is an invariant Seifert surface for $K$, so we can use it to compute the class of $K$ in $\widetilde{\G}^{\Z}$.
With respect to the basis $\{\alpha,\beta,\gamma,\delta\}$ of $H_1(F,\Z)$ the Seifert form and the involution $\rho_*$ are represented by the matrices $A$ and $P$ respectively:
$$A=\begin{pmatrix}
1 & 1 & 0 & 0\\
0 & -1 & -1 & 0\\
0 & -1 & -1 & -1\\
0 & 0 & 0 & -1
\end{pmatrix}$$
$$P=\begin{pmatrix}
1 & 1 & 0 & 0\\
0 & -1 & 0 & 0\\
0 & 0 & -1 & 0\\
0 & 0 & 1 & 1
\end{pmatrix}.$$

The homomorphisms $h$ and $\widetilde{\lk}$ are represented by the covectors:
\begin{gather*}
h=\begin{pmatrix}
0 & -1 & 1 & 0
\end{pmatrix}\\
\widetilde{\lk}=\begin{pmatrix}
-2 & -1 & 1 & 2
\end{pmatrix}.\end{gather*}
One can easily check that $H=\langle \alpha+\delta, 2(\beta-\gamma)-\alpha+\delta\rangle$ is a $\rho_*$-invariant submodule of rank $2$ on which the Seifert form of $F$ vanishes.

Therefore the class of $K$ represents the identity also in the reduced equivariant algebraic concordance group $\widetilde{\G}_r^\Z$.

However, $H=\ker(h)\cap\ker(\widetilde{\lk})=\langle \beta+\gamma, \alpha+\delta\rangle$ has rank $2$ but the Seifert form do not vanishes on $H$, therefore the class of $K$ is nontrivial in $\widetilde{\G}^{\Z}$.
\end{ex}

\subsection{Lower bound on the equivariant slice genus}\label{sect:lower_bounds}
In \cite{miller_powell} the authors obtain a lower bound on the equivariant slice genus of a strongly invertible knot using the Blanchfield form.
Since Miller and Powell's invariant factors through $\widetilde{\G}^\Z$, we can get the same lower bound indirectly.
However, we prove in this section that it is possible to obtain a different lower bound on the equivariant slice genus using the additional information contained in $\widetilde{\G}^\Z$.

\begin{defn}\label{equiv_complx}
Let $\Sigma\subset B^4$ be a properly embedded orientable surface, with boundary a strongly invertible knot $(K,\rho)$. Suppose that $\Sigma$ is invariant under an involution of $B^4$ which extends $\rho$, and denote by $D\cong D^2$ the fixed point set of $\rho$ in $B^4$.
Then intersection $\Sigma\cap D$ consists on an arc joining the two fixed points on $K$ and finite set $\Gamma$ of fixed $S^1$.
We define the \emph{complexity} of $\Sigma$ as
$$c(\Sigma)=\genus(\Sigma)+|\Gamma|.$$
Then, we define the \emph{slice complexity} of a strongly invertible knot $(K,\rho)$ as
$$sc(K,\rho)=\min_{\Sigma} c(\Sigma),$$
where $\Sigma$ ranges among the orientable surfaces in $B^4$ with boundary $K$, invariant under an involution of $B^4$ extending $\rho$.
\end{defn}

\begin{remark}
By Smith theory (see \cite{Bredon1972IntroductionTC}), we have that $|\Gamma|\leq\genus(\Sigma)$. Therefore, for every strongly invertible knot $(K,\rho)$ the follwing inequalities hold
$$\widetilde{g}_4(K)\leq sc(K)\leq 2\cdot\widetilde{g}_4(K).$$
\end{remark}

Let $(K,\rho,h)$ be a directed strongly invertible knot, and let $\Sigma\subset B^4$ an invariant surface for $K$ as in Definition \ref{equiv_complx}.
Denote by $D$ the fixed point set in $B^4$, oriented compatible with the half-axis $h$.
Observe that $D\setminus \Sigma$ can be subdivided into two subsurfaces in a checkerboard fashion, as described in Figure \ref{orientation_fixed_circles}.

\begin{figure}[ht]
\centering
\begin{tikzpicture}

\node[anchor=south west,inner sep=0] at (0,0){\includegraphics[scale=0.5]{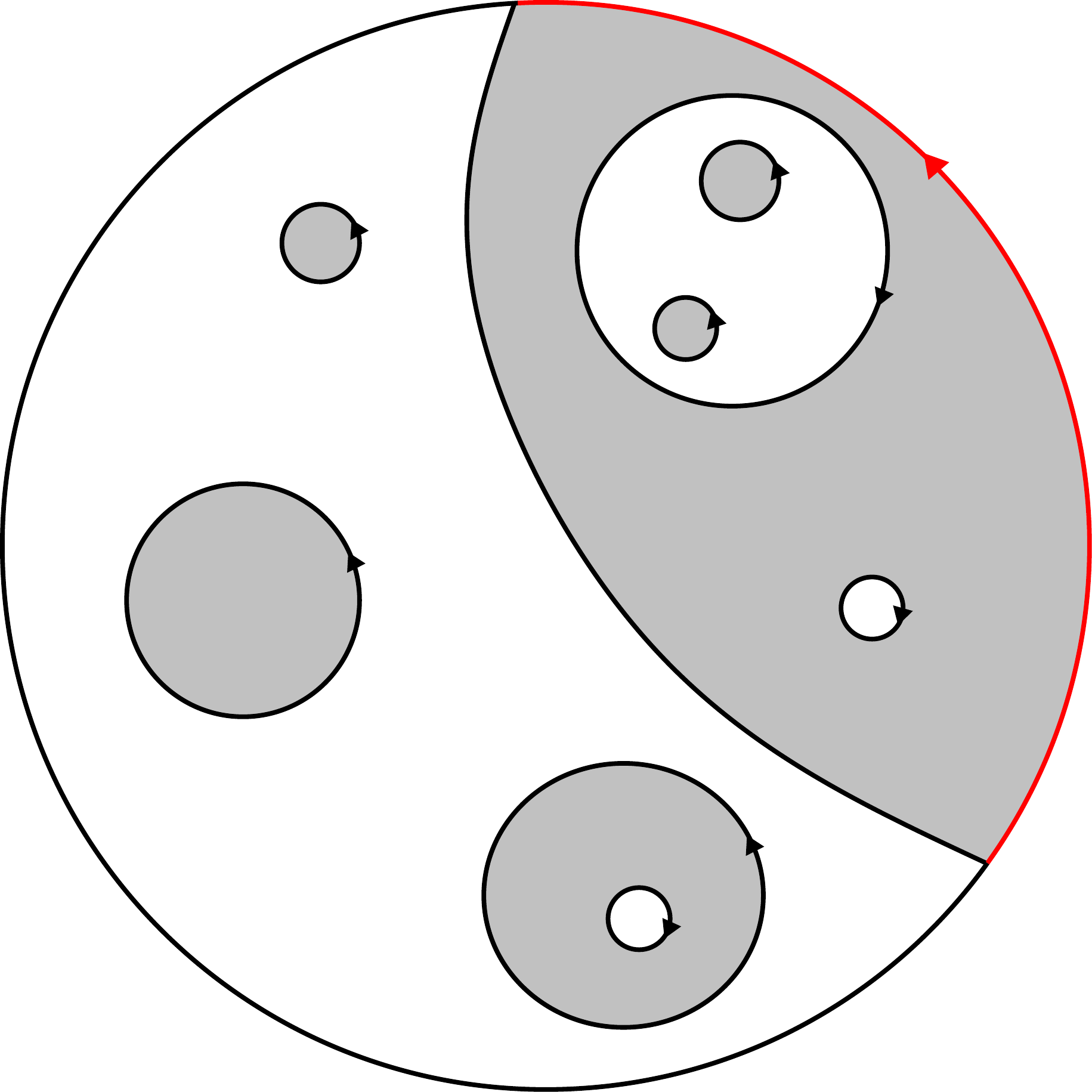}};
\node[label={$h$}] at (8.5,7.5){};
\node[label={$\alpha$}] at (4.8,3.8){};
\end{tikzpicture}
    \caption{An example of how to get the orientation of the fixed point set $D$.}
    \label{orientation_fixed_circles}
\end{figure}

Let $S$ be the subsurface containing the chosen half-axis $h$, and orient every $\gamma\in\Gamma$ and the fixed arc $\alpha$ as the boundary of $S$.

Let $D\times D^2$ be an equivariant tubular neighbourhood of $D$ in $B^4$.
Pick an auxiliary orientation on $\Sigma$ and observe that $\Sigma$ induces on every $\gamma$ a nowhere vanishing section $s_\gamma$ of $D\times D^2$, which we can regard as a map $s_\gamma:\gamma\cong S^1\longrightarrow S^1$.
We call the degree of $s_\gamma$ the \emph{framing} $f(\gamma)\in\Z$ of $\gamma$.
It is easy to see that it does not depend on the auxiliary orientation on $\Sigma$.

Similarly, $\Sigma$ induces on $\alpha$ a nowhere zero section of $D\times D^2$, which we complete to a section on $\alpha\cup h$ using the section induced on $h$ by a band $B\subset S^3$ which gives the $0$-butterfly link of $K$ (see Definition \ref{butterfly_link}).
Call the degree of the associated map $S^1\longrightarrow S^1$ the \emph{framing} $f(\alpha)$ of $\alpha$.

Finally, we say that $\Sigma$ is an \emph{invariant surface type $n$} for $(K,\rho,h)$, where
$$n=f(\alpha)+\sum_{\gamma\in\Gamma}f(\gamma).$$

\begin{prop}\label{remove_fix_pt}
Let $\Sigma\subset B^4$ be an invariant surface of type $n$ for a directed strongly invertible knot $(K,\rho,h)$. Then, there exists an invariant oriented surface $\widehat{\Sigma}\subset B^4$ with boundary $\widehat{L}_b^n(K)$, with no fixed points and such that
$$\genus(\widehat{\Sigma})\leq\begin{cases}c(\Sigma)-1\text{ if } \widehat{\Sigma} \text{ is connected,}\\
c(\Sigma)\text{ if } \widehat{\Sigma} \text{ is not connected}.
\end{cases}$$
\end{prop}

\begin{proof}
On the set of fixed circles $X=\Gamma\cup\{\alpha\cup h\}$ consider the partial order given by the nesting of circles, seen as circles in the fixed disk $D$.
First of all, we want to remove all of the fixed circles. We do so by applying two moves.

\textbf{Move 1}: Suppose there exists a minimal element $\gamma\in\Gamma\subset X$ with framing zero. Let $D_{\gamma}\subset D$ be the disk bounded by $\gamma$. Since $f(\gamma)=0$ the section induced by $\Sigma$ on $\gamma$ of the equivariant tubular neighbourhood of $D$ extends over $D_{\gamma}$ to a nowhere vanishing section, which we can take to be equivariant. Therefore, we can perform an equivariant surgery of $\Sigma$ along $D_{\gamma}$, obtaining a surface $\Sigma'$ of the same type, with less genus and fixed circles.
Replace $\Sigma$ by $\Sigma'$.

\textbf{Move 2}: Let $\gamma\in\Gamma$ be a minimal element with $f(\gamma)\neq 0$ and let $\xi\in X$ be a circle such that there exists an arc $\beta\subset D\setminus\Sigma$ joining $\gamma$ and $\xi$.
Then, we can find an equivariant $D^1\times D^2$ inside an equivariant tubular neighbourhood $D^1\times D^3$ of $\beta$ such that $(\partial D^1)\times D^2\subset \Sigma$. We perform an equivariant surgery along $D^1\times D^2$, obtaining a surface $\Sigma'$ with $\genus(\Sigma')=\genus(\Sigma)+1$. One can check that the circles $\gamma$ and $\xi$ were joined during the surgery into a new fixed circle with framing $f(\gamma)+f(\xi)$. Therefore $\Sigma'$ has the same type of $\Sigma$. Replace $\Sigma$ by $\Sigma'$.

\begin{figure}[ht]
\centering
\begin{tikzpicture}

\node[anchor=south west,inner sep=0] at (0,0){\includegraphics[scale=0.5]{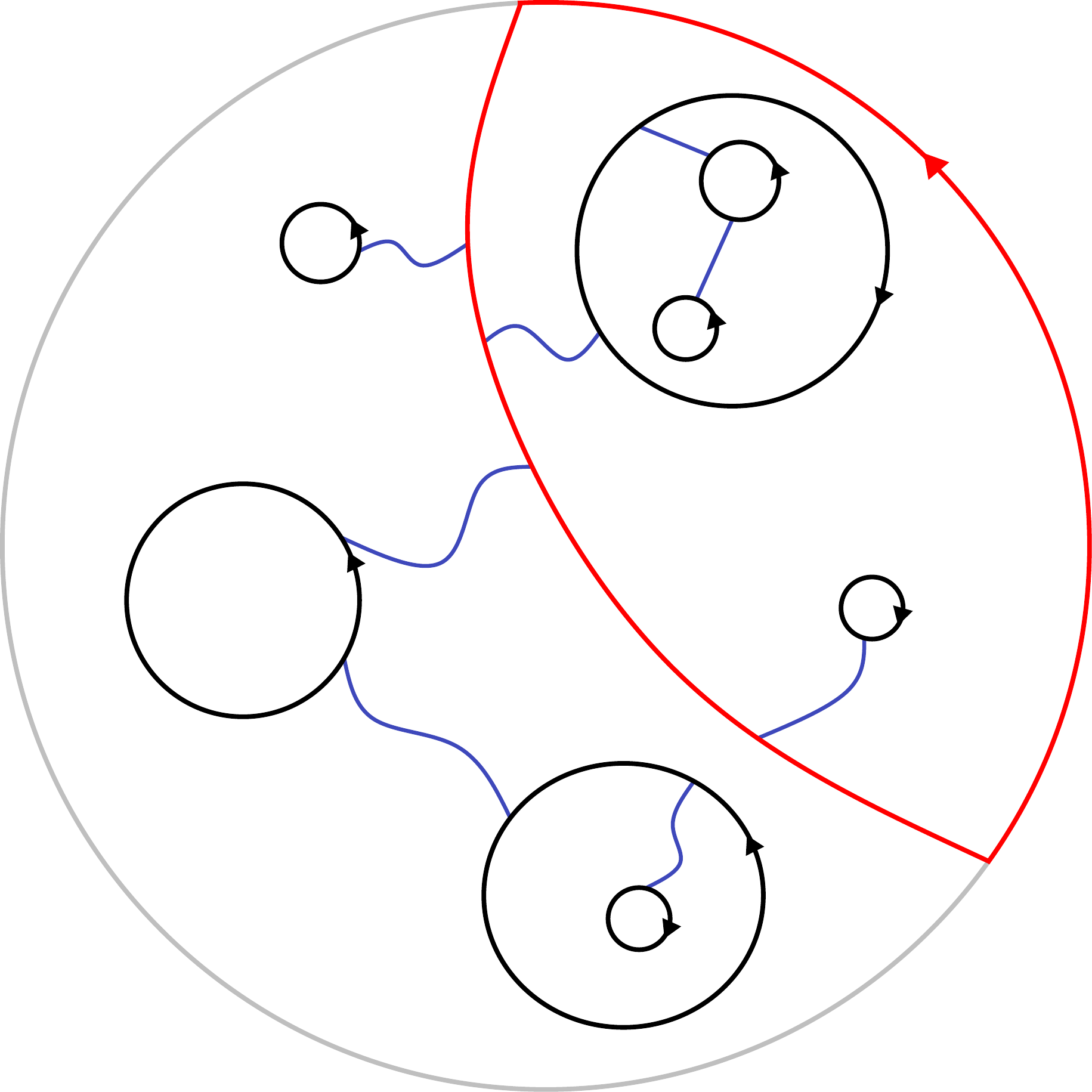}};
\node[label={$h$}] at (8.5,7.5){};
\node[label={$\alpha$}] at (4.8,3.8){};
\end{tikzpicture}
    \caption{An example of choice of the arcs $\beta$, in blue.}
    \label{1surgery}
\end{figure}

Applying Move 1 whenever possible and Move 2 in the other cases, we get an invariant surface $\Sigma'$ of type $n$ with fixed-point set consisting of only one arc.
Observe that we have to apply Move 2 at most $\#\Gamma$ times.
As in Remark \ref{slice_link} we can consider the equivariant band move on $K$ along $h$ that gives $\widehat{L}_b^n(K)$ as an equivariant cobordism $C$ between $K$ and $\widehat{L}_b^n(K)$.
Now glue together $C$ and $\Sigma'$ along $K$, obtaining a surface $\Sigma''$, with $\genus(\Sigma'')\leq c(\Sigma)$.
By construction, the fixed point set of the involution on $\Sigma''$ consists of a single circle, with framing induced by $\Sigma''$ equal to zero. Finally, apply Move 1, obtaining an invariant surface $\widehat{\Sigma}$ with boundary $\widehat{L}_b^n(K)$ and without fixed points.
Observe that if the final Move 1 does not disconnect the surface, then $\genus(\widehat{\Sigma})=\genus(\Sigma'')-1\leq c(\Sigma)-1$. Otherwise $\genus(\widehat{\Sigma})=\genus(\Sigma'')\leq c(\Sigma)$.
\end{proof}

\begin{defn}\label{def:alg_complx}
Let $\Sy=(\theta,\rho,h,\widetilde{\lk})$ be an equivariant Seifert system on $M\cong\Z^{2m}$.
A \emph{partial metabolizer} for $\Sy$ is a $\rho$-invariant submodule $H\subset \ker(h)\cap\ker(\widetilde{\lk})$ such that $\theta_{H\times H}\equiv 0$.
We define the \emph{algebraic complexity} of $\Sy$ as $ac(\Sy)=m-k$, where
$$k=\max\{\rank(H)\,|\, H \text{ is a partial metabolizer of } \Sy\}.$$
\end{defn}

\begin{prop}
The algebraic complexity is constant on the equivariant algebraic concordance classes. Therefore it induces a well-defined map
$$ac:\widetilde{\G}^{\Z}\longrightarrow\mathbb{N}.$$
\end{prop}
\begin{proof}
Let $\mathcal{S}_i=(A_i,P_i,h_i,\widetilde{\lk}_i)$, $i=0,1$ be two equivariant Seifert systems. To prove that the algebraic complexity does not depend on the representative of a class in $\widetilde{\G}^\Z$ it is sufficient to show that if $\Sy_0$ is equivariantly metabolic then $ac(\Sy_0\oplus\Sy_1)=ac(\Sy_1)$.
Let $H_0$ is a metabolizer for $\Sy_0$ and $H_1$ is a maximal rank partial metabolizer for $\Sy_1$ then $H_0\oplus H_1$ is a partial metabolizer for $\Sy_0\oplus\Sy_1$, showing that
$$
ac(\Sy_0\oplus\Sy_1)\leq ac(\Sy_1).
$$
For simplicity, consider the equivariant Seifert systems to be defined over $\Q$ coefficients. It can be seen, as in the proof of Lemma \ref{lemma:inclusion}, that passing from $\Z$ to $\Q$ coefficients does not change the maximal rank of a partial metabolizer.
Denote by $V_i$ the $m_i$-dimensional vector space underlying $\Sy_i$.
Let $H\subset V_0\oplus V_1$ be a maximal partial metabolizer for $\Sy_0\oplus\Sy_1$, $H_0$ be a metabolizer for $\Sy_0$ and $W$ a complement of $H_0$, so that $V_0=H_0\oplus W$. Denote by $k$ the dimension of $H$ and let $\{\alpha_i=(x_i,y_i,z_i)\}_{i=1,\dots,k}$ be a basis of $H$, where $x_i\in H_0$, $y_i\in W$ and $z_i\in V_1$.
Up to base change, we can suppose that $y_1,\dots,y_r$ are linearly independent and that $y_{r+1}=\dots=y_k=0$. Denote by $Y$ the span of $y_1,\dots,y_r$.
Therefore, for $r+1\leq i\leq k$ we have $\alpha_i=(x_i,0,z_i)$. In particular, observe that the subspace spanned by $\{\alpha_i\}_{r+1\leq i\leq k}$ is still invariant under the action of $(P_0\oplus P_1)$.
Repeating the process, we can assume after a change of basis that $z_{r+1},\dots,z_{r+s}$ are linearly independent and that $z_{r+s+1}=,\dots,=z_{k}=0$. Denote by $Z$ the span of $z_{r+1},\dots,z_{r+s}$. Observe that $Z$ is the projection onto the $V_1$ summand of $\langle\alpha_i\;|\;r+1\leq i\leq k\rangle$, hence $P_1(Z)=Z$.
Now for $r+1\leq i,j\leq r+s$ we have that
\begin{gather*}
    0=(A_0\oplus A_1)(\alpha_i,\alpha_j)=A_0(x_i,x_j)+A_1(z_i,z_j)=A_1(z_i,z_j),\\
    0=h_0(x_i)+h_1(z_i)=h_1(z_i),\\
    0=\widetilde{\lk}_0(x_i)+\widetilde{\lk}_1(z_i)=\widetilde{\lk}_1(z_i).
\end{gather*}
In other words, $Z$ is a partial equivariant metabolizer of dimension $s$ for $\Sy_1$. Hence, it is now sufficient to show that $s\geq k-m_0$.

Since for $r+s+1\leq i\leq k$ we have that $\alpha_i=(x_i,0,0)$ are linearly independent, hence $X=\langle x_i\;|\;r+s+1\leq i\leq k\rangle$ is a $(k-s-r)$-dimensional subspace of $H_0$. Moreover, observe that $X\perp Y$, where $\perp$ means orthogonal with respect to the skew-symmetric form $A_0-A_0^t$. On the other hand, $X\perp H_0$ and $H_0\cap Y=0$. Since $A_0-A_0^t$ is non-degenerate we have that
\begin{align*}
    2m_0-(k-r-s)=\dim X^\perp\geq\dim H_0+\dim Y=m_0+r,
\end{align*}
and therefore $s\geq k-m_0$, i.e. the opposite inequality $ac(\Sy_0\oplus\Sy_1)\geq ac(\Sy_1)
$ holds.
\end{proof}

\begin{thm}\label{lower_bound_genus}
Let $(K,\rho,h)$ be a directed strongly invertible knot and let $\Phi(K)\in\widetilde{\G}^\Z$ be its equivariant algebraic concordance class.
Then the following inequality holds
$$2\widetilde{g}_4(K)\geq sc(K)\geq\min_{n\in\Z}ac(\Phi(K)+n\Sy(G)),$$
where $G$ is the invariant Seifert surface of type $-1$ for the unknot in Figure \ref{unknot_surface}.
\end{thm}

\begin{proof}
Let $\Sigma\subset B^4$ be any invariant orientable surface with boundary $K$ and let $n$ be the type of $\Sigma$.
Let $\widehat{\Sigma}\subset B^4$ be the invariant orientable surface with boundary $\widehat{L}_b^n(K)$ and no fixed points obtained from $\Sigma$ by Proposition \ref{remove_fix_pt}.
Take now an invariant Seifert surface $F\subset S^3$ of type $n$ for $K$.
By Proposition \ref{alg_slice_link} there exists a partial metabolizer $H\subset H_1(F,\Z)$, with $\rank H\geq g_F-c(\Sigma)$.
Therefore $c(\Sigma)\geq g_F-\rank H\geq ac(\Sy(F))\geq\min_{n\in\Z}ac(\Phi(K)+n\Sy(G))$.
Taking the minimum over $\Sigma$, we get
$$sc(K)\geq\min_{n\in\Z}ac(\Phi(K)+n\Sy(G)).$$
\end{proof}

\begin{remark}
Since the equivariant slice genus and the slice complexity do not depend on the choice of the direction, one can replace $(K,\rho,h)$ by its antipode $(K,\rho,h')$ in Theorem \ref{lower_bound_genus} to obtain a (potentially) better lower bound.
\end{remark}

\section{Equivariant Gordon-Litherland form}\label{sect:equiv_GL}
In this section, we define a homomorphism from $\widetilde{\G}_r^\Z$ to a simpler group $\widetilde{\W}(\Q)$ of algebraic concordance, namely an equivariant version of the Witt group of $\Q$.
Then we characterize the image of a directed strongly invertible knot in $\widetilde{\W}(\Q)$ in terms of classical Witt invariants. Finally, we prove that the equivariant signature defined in \cite{alfieri2021strongly} factors through $\widetilde{\W}(\Q)$.

\begin{defn}
Let $\F$ be a field. An \emph{equivariant symmetric form} is a pair $(Q,\rho)$ where $Q$ is a symmetric, bilinear, and non-degenerate form on a finite-dimensional $\F$-vector space $V$ and $\rho$ is a $Q$-isometric involution of $V$.
We say that $(Q,\rho)$ is \emph{equivariantly metabolic} if $\dim V$ is even and there exists a half-dimensional $\rho$-invariant subspace $W\subset V$ such that $Q_{|W\times W}\equiv 0$.
\end{defn}

Again, analogously to Definition \ref{orth_sum} and \ref{equiv_alg_conc} we can define a notion of orthogonal sum and concordance between equivariant symmetric forms and define the \emph{equivariant Witt group} $\widetilde{\W}(\F)$ of $\F$ to be the set of equivalence classes of equivariant symmetric forms up to equivariant concordance.

Given an equivariant Seifert form $(\theta,\rho)$ defined over a $\Z$-module $M$, we can define an equivariant symmetric form on $M\otimes_{\Z}\Q$ by $(\theta+\theta^t,\rho)$.

It is immediate to see that this association induces a group homomorphism
$$\widetilde{\mathcal{G}}_r^\Z\longrightarrow\widetilde{\W}(\Q).$$
Denote by $\Phi_W:\C\longrightarrow \widetilde{\W}(\Q)$ the map given by the composition

$$\C\longrightarrow\widetilde{\mathcal{G}}_r^\Z\longrightarrow\widetilde{\W}(\Q).$$

Notice that the map $\Phi_W$ maps the equivariant concordance class of a directed strongly invertible knot $(K,\rho,h)$, to the Witt class of the couple $(\G_F,\rho_*)$, where $F$ is an invariant Seifert surface of type $0$ for $K$ and $(\G_F,\rho_*)$ is the couple given by the Gordon-Litherland form on $H_1(F,\Q)$ and the action induced by $\rho$.

\subsection{A characterization of the equivariant Witt class}
Now we show that given a directed strongly invertible knot $K$, the equivariant Witt class $\Phi_W(K)$ depends only on the (classical) Witt class of $K$ and $\mathfrak{qb}(K)$ (see Definition \ref{butterfly_homomorphisms}).

\begin{remark}
Let $(Q,\rho)$ be an equivariant form over $\Q$ and let $E_{\lambda}$ be the $\lambda$-eigenspace of $\rho$, for $\lambda=\pm1$. Given $v\in E_1$, $w\in E_{-1}$ clearly
$$Q(v,w)=Q(v,-w)=-Q(v,w)\Longrightarrow Q(v,w)=0.$$
Hence, $E_{1}$ and $E_{-1}$ are orthogonal and we can decompose the form as
$$(Q,\rho)=(Q_{|E_1},\id)\oplus(Q_{|E_{-1}},-\id).$$
This gives us an isomophism $$(\pi_+,\pi_-):\widetilde{\W}(\Q)\longrightarrow\W(\Q)\oplus\W(\Q)$$
$$[Q,\rho]\longrightarrow ([Q_{|E_1}],[Q_{|E_{-1}}]).$$
We will denote by $\Phi^{\pm}_W=\pi_{\pm}\circ\Phi_W:\C\longrightarrow\W(\Q)$ the induced homomorphisms. 
\end{remark}

Using the description above of $\widetilde{\W}(\Q)$ we can give a new definition of the equivariant signature.
\begin{defn}\label{equiv_sign_witt}
Denote by $\sigma:\W(\Q)\longrightarrow\Z$ the signature homomorphism.
Define the \emph{equivariant signature} as $\widetilde{\sigma}=(\sigma\circ\Phi^-_W -\sigma\circ\Phi^+_W):\C\longrightarrow\Z$.
\end{defn}

We show now how the invariants we just defined are related to some of the invariants defined in \cite{boyle2021equivariant} and \cite{alfieri2021strongly}.

\begin{defn}
Let $A$ be a non-degenerate symmetric $n\times n$ matrix and let $k$ be a non-zero integer. Define $M_k(A)=k\cdot A$.
Clearly if $A$ is metabolic, $M_k(A)$ is so. Moreover $M_k(A)\oplus M_k(B)=M_k(A\oplus B)$. Therefore, this induces a well-defined homomorphism
$$M_k:\W(\Q)\longrightarrow\W(\Q)$$
$$[A]\longmapsto [M_k(A)].$$
It is immediate to see that $M_k\circ M_k$ is the identity, hence $M_k$ is an isomorphism.
\end{defn}

\begin{lemma}\label{euler_quotient}
Let $F$ be an invariant Seifert surface of type $0$ for a directed strongly invertible knot $(K,\rho,h)$ and let $\widetilde{F}$ be the corresponding Seifert surface for $\widehat{L}_b^0(K)$.
Then the quotient surface $\overline{F}=\widetilde{F}/\rho\subset \overline{S^3}=S^3/\rho$ is a spanning surface for $\qb(K)$ with zero relative Euler number $e(\overline{F})$ (see Definition \ref{euler_number}).
\end{lemma}
\begin{proof}
Pick a representative of the semi-orientation on $\widehat{L}_b^0(K)$ and denote the components of the link by $H$ and $J$.
Let $H^{\widetilde{F}}$ be a nearby longitude of $H$ missing $\widetilde{F}$.
Then the projection $\pi(H^{\widetilde{F}})$ is a longitude of $\qb(K)$ missing $\overline{F}$.
In order to show that $e(\overline{F})=\lk(\qb(K),\pi(H^{\widetilde{F}}))=0$ it is sufficient to prove that $[H^{\widetilde{F}}]=0\in H_1(S^3\setminus H,\Z)$.
Since $H^{\widetilde{F}}$ is disjoint from $\widetilde{F}$, we have that $\lk(H^{\widetilde{F}},H)+\lk(H^{\widetilde{F}},J)=0$.
By definition of $0$-butterfly link we have that $\lk(H^{\widetilde{F}},J)=\lk(H,J)=0$, therefore $\lk(H^{\widetilde{F}},H)=0$.
In other words $[H^{\widetilde{F}}]=0\in H_1(S^3\setminus H,\Z)$.
\end{proof}

\begin{prop}
Let $(K,\rho,h)$ be a directed strongly invertible knot. Then
$$\Phi_W^+(K)=M_2(\phi_W(\mathfrak{qb}(K)),$$
where $\phi_W(\mathfrak{qb}(K))$ is the Witt class of $\mathfrak{qb}(K)$.
\end{prop}
\begin{proof}
Let $F$ be an invariant Seifert surface of type $0$ for $K$ and let $\widetilde{F}$ be the corresponding Seifert surface for $\widehat{L}_b^0(K)$. First of all, observe that $F$ can be obtained from $\widetilde{F}$ by attaching an equivariant band $B$. Since $\rho$ reverses the orientation of the core of $B$, it is not difficult to see that the dimension of the $(-1)$-eigenspace of $\rho_*$ increases by one going from $H_1(\widetilde{F},\Q)$ to $H_1(F,\Q)$. Hence, the $1$-eigenspace of $\rho_*$ is fully contained in $H_1(\widetilde{F})$. Let $\pi:(S^3,\widetilde{F})\longrightarrow(\overline{S^3},\overline{F})$ be the quotient projection, given by the action of $\rho$. The quotient surface $\overline{F}$ is a spanning surface for $\qb(K)$.
Observe that the quotient projection $\pi$ is a $2$-fold covering $\widetilde{F}\longrightarrow \overline{F}$. Take now an oriented curve $c$ in $\overline{F}$, representing a class in $H_1(\overline{F})$, and lift it to a class $\tr(c)=\pi^{-1}(c)\in H_1(\widetilde{F})$.
This defines a transfer homomorphism (see \cite{Bredon1972IntroductionTC} for details) $\tr:H_1(\overline{F})\longrightarrow H_1(\widetilde{F})$. By construction $\rho_*(\tr(c))=\tr(c)$, i.e. the image of the transfer map is contained in the $1$-eigenspace $E_1$ of $\rho_*$. The composition $\pi_*\circ\tr$ is given by $$\pi_*\circ \tr:H_1(\overline{F})\longrightarrow H_1(\overline{F})$$
$$c\longmapsto 2c,$$
hence $\tr$ is injective. Moreover $\tr$ is clearly surjective on $E_1$: given a $\rho$-invariant class $d\in E_1$, we can project it by $\pi_*$ and lift it again, showing that $2d\in \im(\tr)$.
Finally, we show that the transfer map behaves well with respect to the Gordon-Litherland form. Given $c,d\in H_1(\overline{F})$ let $S$ be an oriented surface in $\overline{S^3}$ with $\partial S=\widetilde{d}$, where $\widetilde{d}$ is $d$ pushed out of $\overline{F}$ ``in both directions simultaneously'' (as in the definition of the Gordon-Litherlan form). In this way we have that $\G_{\overline{F}}(c,d)=\lk(c,\widetilde{d})=\#S\cap c$. Up to a small isotopy, we can suppose $S$ transverse to the branching locus. The lift $\pi^{-1}(S)$ is an oriented surface in $S^3$ with boundary $\widetilde{\tr(d)}$ and we can use it to calculate $$\lk(\tr(c),\widetilde{\tr(d)})=\#(\pi^{-1}(S)\cap\tr(c))=2(\#S\cap c).$$
It follows that the Gordon-Litherland forms are related by $$\G_{\widetilde{F}}(\tr(c),\tr(d))=2\cdot\G_{\overline{F}}(c,d).$$
Therefore the transfer map gives an isometry
$$\tr:(H_1(\overline{F}),2\cdot\G_{\overline{F}})\longrightarrow (E_1,\G_{\widetilde{F}}).$$
Finally, by Lemma \ref{euler_quotient} we have that $e(\overline{F})=0$ and hence by Proposition \ref{GL_refined} that the Gordon-Litherland form on $\overline{F}$ represents the Witt class of $\qb(K)$.
\end{proof}

As an immediate consequence, we get the following corollary.
\begin{cor}\label{qb_phi}
Let $K$ be a directed strongly invertible knot and let $A$ and $B$ be symmetric matrices representing the (non-equivariant) Witt classes of $\mathfrak{qb}(K)$ and $K$ respectively.
Then the equivariant Witt class of $K$ is represented by the couple
$$\Phi_W(K)=\left[\begin{pmatrix}
2A&0&0\\
0&-2A&0\\
0&0&B
\end{pmatrix}, \begin{pmatrix}
\id&0&0\\
0&-\id&0\\
0&0&-\id
\end{pmatrix}\right].$$
\end{cor}

\subsection{The equivariant signature}\label{equiv_sign_sect}
We recall now the definition of equivariant signature introduced by Alfieri and Boyle \cite{alfieri2021strongly} and we prove that it is equivalent to the one in Definition \ref{equiv_sign_witt}.

\medskip
Given a knot $K\subset S^3$ we denote by $\Sigma(K)$ the $2$-fold cover of $S^3$ branched over $K$.
Given a properly embedded and connected surface $F\subset B^4$, we denote by $\Sigma(F)$ the $2$-fold cover of $B^4$ branched over $F$, and by $\tau$ the covering transformation of $\Sigma(F)$.

\begin{lemma}\cite[Proposition 12]{boyle2021equivariant}\label{lift}
Let $\rho$ be an orientation preserving involution of $B^4$ such that $\Fix(\rho)$ is a $2$-disk $D$. Let $F\subset B^4$ be a properly embedded and connected $\rho$-invariant surface on which $\rho$ acts non-trivially. Then, there exists a lift $\widetilde{\rho}$ of $\rho$, i.e. the following diagram commutes
\begin{center}
\begin{tikzcd}

\Sigma(F)\ar[r,"\widetilde{\rho}"]\ar[d,"\pi"]&\Sigma(F)\ar[d,"\pi"]\\
B^4\ar[r,"\rho"]&B^4.
\end{tikzcd}
\end{center}
In fact, there exist exactly two such lifts, namely $\widetilde{\rho}$ and $\tau\widetilde{\rho}$.
\end{lemma}

Now let $(K,\rho,h)$ be a directed, strongly invertible knot and let $F\subset B^4$ be a properly embedded connected surface with $\partial F=K$ (not necessarily orientable), invariant with respect to some extension of $\rho$ to $B^4$ (which we still denote by $\rho$).

Before introducing the equivariant signature it is useful to better describe the fixed point set of the lifts of $\rho$ given by Lemma \ref{lift}. We do so in the following remark
\begin{remark}\label{fixed_pt_lift}
Let $D$ be the fixed point disk of $B^4$. The intersection $D\cap F=\Fix(\rho_{|F})$ is the disjoint union of an arc joining the fixed point of $K$ and a finite number of $S^1$ and isolated points.
Take $x\in\Fix(\widetilde{\rho})$ and observe that $\rho\circ\pi(x)=\pi\circ\widetilde{\rho}(x)=\pi(x)$, therefore $\Fix(\widetilde{\rho})\subseteq\pi^{-1}(D)$.
Moreover, note that $\widetilde{\rho}\circ\tau(x)=\tau\circ\widetilde{\rho}(x)=\tau(x)$, i.e. $\Fix(\widetilde{\rho})$ is $\tau$-invariant.
Take now $x\in \pi^{-1}(F\cap D)$. Then $\widetilde{\rho}(x)\in\pi^{-1}(\rho\circ\pi(x))=\{x,\tau(x)\}$ and since $\tau(x)=x$ we have that $x\in\Fix(\widetilde{\rho})$, i.e. $\pi^{-1}(F\cap D)$ is fixed pointwise by $\widetilde{\rho}$.
Let $C_1,\dots,C_n$ the connected components of $D\setminus F$. Since $\rho_{|C_i}$ is the identity, $\widetilde{\rho}$ and $\tau\circ\widetilde{\rho}$ act either as the identity or as $\tau$ on $\pi^{-1}(C_i)$. Therefore, exactly one of the lifts fixes pointwise the preimage of $C_i$, while the other one has no fixed point in $\pi^{-1}(C_i)$.
Let $C_i$ and $C_j$ be \emph{adjacent} components, i.e. separated by a circle or an arc in $D\cap F$. Then, $\pi^{-1}(C_i)$ and $\pi^{-1}(C_j)$ cannot be both fixed pointwise by $\widetilde{\rho}$. Otherwise, $\overline{\pi^{-1}(C_i)}$ and $\overline{\pi^{-1}(C_j)}$ would be two fixed surfaces, both contained in $\Fix(\widetilde{\rho})$ and intersecting in a non-trivial way in their interior and this would imply that $\Fix(\widetilde{\rho})$ has a component which is not a manifold.
Therefore, if we decompose $D\setminus F=A\sqcup B$, where $A$ and $B$ are union of non-adjacent components, we have that the fixed point sets of the two lifts of $\rho$ are respectively $\pi^{-1}(\overline{A})$ and $\pi^{-1}(\overline{B})$.
\end{remark}

Observe that by Remark \ref{fixed_pt_lift} exactly one lift $\widetilde{\rho}$ of $\rho$ to $\Sigma(F)$ fixes pointwise $\widetilde{h}=\pi^{-1}(h)$.
The fixed point set of $\widetilde{\rho}$ is the disjoint union of a (eventually disconnected) surface $\Delta$, with $\partial\Delta=\widetilde{h}$, and a finite set of points.
Recall now that the $0$-butterfly link $L_b^0(K)$ is obtained by performing a band move on $K$ along a band parallel to $h$, in such a way that the linking number between the components of $L_b^0(K)$ is zero. Let $\gamma$ be one of the arcs of this band parallel to $h$. Since the endpoints of $\gamma$ meet the branching set, its preimage $\widetilde{\gamma}$ in $\Sigma(F)$ is a closed curve.
Given a perturbation $\Delta'$ of $\Delta$ with $\partial\Delta'=\widetilde{\gamma}$ we define the \emph{relative Euler number} $e(\Delta,\widetilde{\gamma})$ as the algebraic intersection $\#(\Delta\cap\Delta')$.

\begin{defn}\cite{alfieri2021strongly}\label{equiv_signature}
The \emph{equivariant signature} of $(K,\rho,h)$ is defined as
$$\widetilde{\sigma}(K)=\sigma(\Sigma(F),\widetilde{\rho})-e(\Delta,\widetilde{\gamma}),$$
where $\sigma(\Sigma(F),\widetilde{\rho})$ is the $g$-signature (see \cite{alfieri2021strongly} or \cite{Gordon1986}) of the pair $(\Sigma(F),\widetilde{\rho})$.
\end{defn}

Using the $G$-signature Theorem \cite{Gordon1986}, Alfieri and Boyle prove that the equivariant signature is a well-defined invariant for equivariant concordance and in particular defines a homomorphism
$$\widetilde{\sigma}:\C\longrightarrow\Z.$$

\begin{remark}
Actually, Alfieri and Boyle \cite{alfieri2021strongly} define the equivariant signature slightly differently, exchanging the role of the two half-axes $h$ and $h'$. It is immediate to check that our definition of equivariant signature for the directed strongly invertible knot $(K,\rho,h)$ coincides with their definition for the antipode $(K,\rho,h')=a(K,\rho,h)$.
Hence the two invariants are essentially the same. However, it is easier to relate Definition \ref{equiv_signature} to the equivariant algebraic concordance group (see Theorem \ref{equivalent_definition}).
\end{remark}

In \cite[Section 6]{alfieri2021strongly} the authors explain how to easily compute the relative Euler number for the equivariant pushoff of a spanning surface in $B^4$.
Using the following proposition it is possible to easily compute the equivariant signature from an equivariant spanning surface.
\begin{prop}\cite[Proposition 13]{boyle2021equivariant}\label{lattice_isom_equiv}
Let $(K,\rho,h)$ be a directed strongly invertible knot in $S^3$. Let $F$ be a connected spanning surface for $K$, with $\rho(F)=F$.
We still denote by $\rho$ the radial extension of the involution to $B^4$. Let $\widehat{F}$ be the surface obtained by equivariantly pushing the interior of $F$ in $B^4$ and denote by $\widetilde{\rho}$ the preferred lift of $\rho$ to $\Sigma(\widehat{F})$.
Then under the identification $(H_1(F),\G_F)\cong (H_2(\Sigma(\widehat{F})),Q)$ the map of lattices $\widetilde{\rho}_*:(H_2(\Sigma(\widehat{F})),Q)\longrightarrow(H_2(\Sigma(\widehat{F})),Q)$ is equivalent to:
\begin{itemize}
    \item $\rho_*:(H_1(F),\G_F)\longrightarrow(H_1(F),\G_F)$ if $h\not\subset F$,
    \item $-\rho_*:(H_1(F),\G_F)\longrightarrow(H_1(F),\G_F)$ if $h\subset F$,
\end{itemize}
\end{prop}

\begin{thm}\label{equivalent_definition}
The equivariant signature introduced in Definition \ref{equiv_sign_witt} coincides with the one given in Definition \ref{equiv_signature}.
\end{thm}
\begin{proof}
Let $F$ be an invariant Seifert surface of type $0$ for a directed strongly invertible knot $(K,\rho,h)$. According to Definition \ref{equiv_sign_witt}, $\widetilde{\sigma}(K)$ is the equivariant signature of $(H_1(F),\G_F,-\rho_*)$. By Lemma \ref{lift} and Proposition \ref{lattice_isom_equiv} this quantity coincides with the $g$-signature of the pair $(\Sigma(\widehat{F}),\widetilde{\rho})$, where $\Sigma(\widehat{F})$ is the $2$-fold cyclic cover branched over a copy $\widehat{F}$ of $F$ radially pushed into $B^4$ and $\widetilde{\rho}$ is the preferred lift of the radial extension of $\rho$ to $B^4$. Hence, it is sufficient to prove that the relative Euler number vanishes.
Let $\gamma$ be a parallel copy of $h$ on $F$. Since cutting $F$ along $h$ produces an equivariant Seifert surface for $\widehat{L}_b^0(K)$, the lift $\widetilde{\gamma}$ of $\gamma$ in $\Sigma(\widehat{F})$ is the canonical longitude of $\widetilde{h}$. 
Let $D,D'$ be the traces of $h$ and $\gamma$ respectively along the radial isotopy that pushes the interior of $F$ in $B^4$.
Since $\widehat{F}$ is obtained from $F\subset S^3$, one can see that $\Fix(\rho_{\widehat{F}})$ consists solely on an arc joining the fixed points of $K$.

Then by Remark \ref{fixed_pt_lift} the fixed point set of $\widetilde{\rho}$ consists of the lift $\Delta$ of $D$. The lift $\Delta'$ of $D'$ is a perturbation of $\Delta$ such that $\partial \Delta'=\widetilde{\gamma}$, and since they are disjoint we have $$e(\Delta,\widetilde{\gamma})=\#(\Delta\cap\Delta')=0$$
i.e. the relative Euler number vanishes.
\end{proof}

\begin{remark}\label{remark:direction}
As a consequence of Theorem \ref{equivalent_definition} and Corollary \ref{qb_phi} we obtain the following formula
$$\widetilde{\sigma}(K)=\sigma(K)-2\sigma(\mathfrak{qb}(K))$$
for the equivariant signature of a directed strongly invertible knot $K$ in terms of classical signatures.
\end{remark}

\begin{remark}\label{equiv_signature_direction}

As shown by Alfieri and Boyle \cite[Proposition 7.3]{alfieri2021strongly}, the equivariant signature depends on the choice of the half-axis for a strongly invertible knot. For example, they show in the proof of Proposition 7.2 that $\widetilde{\sigma}(7_4b^+\widetilde{\#}m7_4b^-)\neq 0$.

Since the equivariant algebraic concordance homomorphism $\Psi:\C\longrightarrow\AC$ defined in \cite{miller_powell} does not distinguish the choice of half-axis, we get that $7_4b^+\widetilde{\#}m7_4b^-$ has trivial image in $\AC$.

On the other hand, since $\widetilde{\sigma}$ factors through $\widetilde{\G}_r^\Z$, we get that the image of $7_4b^+\widetilde{\#}m7_4b^-$ is non trivial in $\widetilde{\G}_r^\Z$.

\end{remark}

\section{The structure of the equivariant algebraic concordance group}\label{sect:grp_str}
In this section, we explore the structure of the reduced equivariant algebraic concordance group $\widetilde{\G}_r^{\Z}$. In order to do so, we introduce the notion of \emph{symmetric structure} and we describe the equivalence between equivariant Seifert forms over $\Q$ and symmetric structures.

\begin{lemma}\label{lemma:inclusion}
The natural homomorphism $\widetilde{\G}_r^\Z\longrightarrow\widetilde{\G}_r^\Q$ given by the extension of coefficients is injective.
\end{lemma}
\begin{proof}
Let $(\theta,\rho)$ be an equivariant Seifert form over a free $\Z$-module $M$, and denote by $(\theta_\Q,\rho_\Q)$ its extension over $M\otimes\Q$. Suppose $H\subset M\otimes\Q$ is a $\rho_\Q$-invariant metabolizer for $\theta_\Q$ and let $H_\Z=H\cap M$. Then clearly $H_\Z$ is a $\rho$-invariant metabolizer for $\theta$.
\end{proof}

The lemma above ensures that no information is lost by considering rational coefficients instead of integral ones. We focus now on determining the group structure of $\widetilde{\G}_r^\Q$. In the following, we will implicitly consider the equivariant Seifert forms to be over $\Q$.

\begin{defn}
A \emph{symmetric structure} is a triple $(V,\beta,S)$, where
\begin{itemize}
    \item $V$ is a $\Q$-vector space,
    \item $\beta$ is a bilinear, symmetric and non-degenerate form on $V$,
    \item $S:V\longrightarrow V$ is a linear isomorphism which is self-adjoint with respect to $\beta$.
\end{itemize}

We say that $(V,\beta,S)$ is \emph{metabolic} if $\dim V$ is even and there exists a $S$-invariant half-dimensional subspace $W\subset V$ on which $\beta$ is identically zero.

We define the \emph{orthogonal sum} of two symmetric structures $\mathcal{\Sy}_i=(V_i,\beta_i,S_i)$, $i=1,2$ as
$$
\mathcal{\Sy}_1\oplus\mathcal{\Sy}_2=(V_1\oplus V_2,\beta_1\oplus\beta_2,S_1\oplus S_2).
$$

We say that $\Sy_1$ and $\Sy_2$ are \emph{concordant} if $-\Sy_1\oplus\Sy_2$ is metabolic, where $-\Sy_1=(V_1,-\beta_1,S_1)$.
\end{defn}

\begin{defn}
We define the group $\G_{sym}$ of symmetric structures as the quotient of the set of symmetric structures up to concordance, endowed with the operation of orthogonal sum.
\end{defn}

Let $(V,\theta,\rho)$ be an equivariant Seifert form. Let $T$ be the endomorphism of $V$ given by the composition
\begin{center}
    \begin{tikzcd}  V\ar[r,"\theta-\theta^t"]&V^*\ar[r,"(\theta+\theta^t)^{-1}"]&V,
    \end{tikzcd}
\end{center}
and let $\beta$ be the bilinear form $\theta+\theta^t$.
It is easy to check that the following facts hold:
\begin{itemize}
    \item $T$ is anti-self-adjoint with respect to $\beta$,
    \item $\rho\circ T+T\circ\rho=0$,
    \item $\theta(x,y)=\frac{1}{2}\left(\beta(x,y)+\beta(Tx,y)\right)$.
\end{itemize}
In particular, the last property implies that any subspace $H\subset V$ is a $\rho$-invariant metabolizer for $\theta$ if and only if is a $\langle T,\rho\rangle$-invariant metabolizer for $\beta$.

Denote now by $V_{\pm}$ the eigenspace of $\rho$ relative to $\pm1$.
Since $\rho\circ T+T\circ\rho=0$, we have that $T(V_\pm)=V_\mp$, and that $V_\pm$ is $T^2$-invariant.

Suppose now $H\subset V$ is a $\langle T,\rho\rangle$-invariant metabolizer for $\beta$, and let $H_{\pm}=H\cap V_{\pm}$.
Since $\rho$ is a $\beta$-isometry, $V_+$ and $V_-$ are $\beta$-orthogonal and hence $H_{\pm}$ is a metabolizer for $\beta_{|V_{\pm}}$.
Since $H$ is $T$-invariant and $T$ and $\rho$ anticommute, we easily deduce that $T(H_{\pm})=H_{\mp}$ and hence that $H_\pm$ is $T^2$-invariant.
Viceversa, let $H_+\subset V_+$ be a $T^2$-invariant metabolizer for $\beta_{|V_+}$. Since $T$ is anti-self-adjoint with respect to $\beta$, it is immediate to see that $H=H_+\oplus T(H_+)$ is a $\langle T,\rho\rangle$-invariant metabolizer for $\beta$.

\begin{thm}
Using the notation above, we have that the following map from equivariant Seifert forms to symmetric structures
$$
(V,\theta,\rho)\longmapsto(V_+,\beta_{|V_+}, T^2_{|V_+})
$$
induces an isomorphism between the equivariant algebraic concordance group $\widetilde{\G}_r^\Q$ and the group of symmetric structures $\G_{sym}$.
\end{thm}
\begin{proof}

It is immediate to see that the map above is compatible with the orthogonal sum. From the discussion above, we have that $(V,\theta,\rho)$ is metabolic if and only if $(V_+,\beta_{|V_+}, T^2_{|V_+})$ is metabolic. Therefore, this map induces a well-defined and injective homomorphism between the two groups.
Viceversa, given a symmetric structure $(V,\beta,S)$, we can construct an equivariant Seifert form on $V\oplus V$, as follows:
$$
\theta=\frac{1}{2}\begin{pmatrix}
    \beta&\beta\cdot S\\
    -\beta\cdot S&-\beta\cdot S
\end{pmatrix},
$$
$$
\rho=\begin{pmatrix}
    \id&0\\
    0&-\id
\end{pmatrix}.
$$
It is easy to check that $(V\oplus V,\theta,\rho)$ has $(V,\beta,S)$ as associated symmetric structure.
\end{proof}

Let $\theta$ be a Seifert form on a space $V$. Recall that the Alexander polynomial of $\theta$ is defined as $\Delta_\theta(t)=\det(A-tA^t)$, where $A$ is a matrix representing $\theta$ with respect to some basis of $V$ (which is well defined up to squares in $\Q$).

\begin{defn}\label{def:poly_transf}
Let $p(t)\in\Q[t,t^{-1}]$ be a width $2d$ \emph{symmetric} polynomial (i.e. such that $p(t)=p(-t)$). We denote by $\delta(p)(s)\in \Q[s]$ the polynomial obtained by the following substitution
$$
\delta(p)(s)=\left(\left(1-\lambda^2\right)^d p\left(\frac{1+\lambda}{1-\lambda}\right)\right)_{|s=\lambda^2}.
$$
Observe that $\delta$ is multiplicative, meaning that $\delta(p\cdot q)=\delta(p)\cdot\delta(q)$. Moreover $\delta$ admits an inverse: given $q(s)\in\Q[s]$ a polynomial of degree $d$ we can define
$$
\delta^{-1}(q)(t)=\frac{(t+1)^{2d}}{(4t)^d}q\left(\left(\frac{t-1}{t+1}\right)^2\right),
$$
and it is not difficult to see that $\delta^{-1}$ is the inverse of $\delta$.
\end{defn}

\begin{remark}\label{rem:poly_transf}
Let $(V,\theta,\rho)$ be an equivariant Seifert form, with Alexander polynomial $\Delta_\theta(t)$. Let $(V_+,\beta,S)$ be the associated symmetric structure.
Then, from the discussion above, we have that up to units in $\Q$ and factors $(s-1)$, the characteristic polynomial of $S$ is given by $\delta(\Delta_\theta)(s)$.
In the following, given a directed strongly invertible knot, we will denote by $\delta_K(s)$ the polynomial obtained by the formula in Definition \ref{def:poly_transf} applied to the Alexander polynomial $\Delta_K(t)$ of $K$ (which is well defined if we require that $\Delta_K(t)=\Delta_K(t^{-1})$).
\end{remark}

Let $(V,\beta,S)$ be a symmetric structure and $p(s)\in\Q[s]$ be an irreducible polynomial. We define the $p$-component of $V$ as
$$V_p=\bigcup_{N>0}\ker(p(S)^N).$$
Clearly, $V_p$ is an $S$-invariant subspace of $V$.

\begin{lemma}\label{lemma:pq-orth}
Let $p$ and $q$ be distinct irreducible polynomials. Then $V_p$ and $V_q$ are orthogonal with respect to $\beta$.
\end{lemma}
\begin{proof}
Let $N>0$ be big enough so that $V_q=\ker(q(S)^N)$. Since $p$ is irreducible and $p\neq q$, we have that the restriction of $q(S)^N$ is an isomorphism from $V_p$ to itself. Given now $v\in V_p$ and $w\in V_q$, there exists $v'\in V_p$ such that $v=q(S)^Nv'$. We compute now
$$
\beta(v,w)=\beta(q(S)^Nv',w)=\beta(v',q(S)^Nw)=\beta(v',0)=0,
$$
therefore $V_p$ and $V_q$ are orthogonal.
\end{proof}

In particular, the restriction of $\beta$ on $V_p$ is non-degenerate and hence $(V_p,\beta_{|V_p},S_{|V_p})$ is a symmetric structure.
The following theorem is an immediate consequence of the lemma above.

\begin{thm}\label{thm:p_decomp}
Let $\mathcal{F}\subset\Q[s]$ be the set of irreducible and monic polynomials different from $q(s)=s$. Then the group $\G_{sym}$ splits as
$$\G_{sym}=\bigoplus_{p\in\mathcal{F}}\G^p_{sym},$$ 
where $\G^p_{sym}$ is the subgroup of $\G_{sym}$ determined by symmetric structures with characteristic polynomial given by a power of $p(s)$.
In particular, the projection of a class $(V,\beta,S)\in\G_{sym}$ onto $\G^p_{sym}$ is given by $(V_p,\beta_{|V_p},S_{|V_p})$.
\end{thm}

Therefore, it is sufficient to study the summands $\G_{sym}^p$ separately.

\begin{prop}\label{prop:exp_red}
Let $(V,\beta,S)\in\G_{sym}^p$. Then $(V,\beta,S)$ is concordant to $(\overline{V},\overline{\beta},\overline{S})$ so that $\overline{V}=\ker p(\overline{S})$.
\end{prop}

\begin{proof}
Let $N\geq 0$ be the least integer so that $V=\ker(p(S)^N)$ and suppose $N\geq 2$. Let $W=\im p(S)^{N-1}$ and let $W^{\perp}$ be its orthogonal. Clearly, $W$ is invariant under $S$, and so is $W^\perp$.
Observe that for $v,w\in V$
$$
\beta(p(S)^{N-1}v,p(S)^{N-1}w)=\beta(v,p(S)^{2N-2}w)=\beta(v,0)=0,
$$
since $2N-2\geq N$ for $N\geq 2$. Therefore $W\subset W^\perp$, and by definition $W$ is the radical of $\beta_{|W^\perp}$. Hence $\beta$ induces a symmetric non-degenerate form $\overline{\beta}$ on the quotient $\overline{V}=W^\perp/W$. Similarly, since $W$ and $W^\perp$ are $S$-invariant, we have an induced map $\overline{S}$ on the quotient.

It is easy to check that $(\overline{V},\overline{\beta},\overline{S})$ is again a symmetric structure, and that by construction $\overline{V}=\ker(p(\overline{S})^{N-1})$

Consider $H=\{(w,[w])\in V\oplus \overline{V}\;|\;w\in W^\perp\}$. It is easy now to check that $H$ is a metabolizer for $(V,-\beta,S)\oplus(\overline{V},\overline{\beta},\overline{S})$, showing that $(V,\beta,S)$ and $(\overline{S},\overline{\beta},\overline{S})$ are concordant.
\end{proof}

\begin{thm}\label{thm:p_comp_witt}
Let $p(s)\in\Q[s]$ be an irreducible polynomial and consider the field $\F=\Q[s]/(p(s))$. Then, we have a natural isomorphism
$$\G^p_{sym}\longrightarrow W(\F).$$
\end{thm}

\begin{proof}
Take $(V,\beta,S)\in\G^p_{sym}$.
By Proposition \ref{prop:exp_red} we can suppose that $V=\ker(p(S))$.
Therefore $V$ is naturally a $\F$-vector space.
Denote by $\tr:\F\longrightarrow\Q$ the trace map.
Since the symmetric $\Q$-bilinear form
$$
\F\times\F\longrightarrow\Q
$$
$$
(x,y)\longmapsto\tr(x\cdot y)
$$
is non degenerate, for every $u,v\in V$ there exists a unique $e\in\F$ such that for all $f\in\F$
$$
\beta(f\cdot u,v)=\tr(e\cdot f).
$$
This correspondence defines a symmetric $\F$-bilinear form $[-,-]_\beta$ on $V$ such that
\begin{center}
\begin{tikzcd}
    &\F\ar[d,"\tr"]\\
    V\times V\ar[ur,"{[-,-]}_\beta"]\ar[r,"\beta",swap]&\Q
\end{tikzcd}
\end{center}
is a commutative diagram.

Now let $H$ be a $\Q$-subspace of $V$. Then $H$ is $S$-invariant if and only if is a $\F$-subspace. 
Moreover, it is easy to see that $\beta$ vanishes identically on a $\F$-subspace $H$ if and only if $[-,-]_\beta$ does.
Therefore, the homomorphism induced by this correspondence is an isomorphism.
\end{proof}

We can now summarize Theorems \ref{thm:p_decomp} and \ref{thm:p_comp_witt} in the following theorem on the structure of $\widetilde{\G}_r^\Q$.
\begin{thm}\label{thm:grp_str}
The reduced equivariant algebraic concordance group $\widetilde{\G}_r^\Q$ is isomorphic to
$$
\widetilde{\G}_r^\Q\cong\bigoplus_{\substack{p(s)\neq s\\\text{p irreducible}}}W(\Q[s]/(p(s))).
$$
\end{thm}

Using classical results on the Witt groups of finite extension of $\Q$ (see \cite{milnor1973symmetric} for details), one can check that Theorem \ref{thm:grp_str} implies that
$$
\widetilde{\G}_r^\Q\cong\Z^\infty\oplus\Z/2\Z^\infty\oplus\Z/4\Z^\infty\oplus\Z/8\Z^\infty,
$$
where $G^\infty$ stands for the direct sum of countable many copies of $G$.

\begin{remark}\label{rem:non_surj}

It would be very interesting to know whether there exists any directed strongly invertible knot $K$ whose image has order $8$ in $\widetilde{\G}_r^\Q$. However, we would like to point out that several of the summands in Theorem \ref{thm:grp_str} have trivial intersection with $\widetilde{\G}_r^\Z$: it is not difficult to check that many polynomials do not appear as factors of $\delta_K(s)$ for any knot $K$.

For example, suppose that $p(s)=s-10$ appears as a factor of $\delta_K(s)=p(s)q(s)$ for some knot $K$ and some other polynomial $q$. Following Remark \ref{rem:poly_transf} we compute the Alexander polynomial of $K$ (up to units) as $\delta^{-1}(\delta_K(s))(t)=\delta^{-1}(p)(t)\cdot\delta^{-1}(q)(t)=\frac{1}{4}(-9t-22-9t^{-1})\cdot \delta^{-1}(q)(t)$. Therefore, $\Delta_K(t)$ would be divisible by the primitive polynomial $\Delta(t)=9t+22+9t^{-1}$. Since $\Delta(1)=40$, we get a contradiction with the fact that $\Delta_K(1)=\pm1$.
\end{remark}

We conclude this section with a consequence of Theorem \ref{thm:grp_str}, which can be interpreted as an \emph{equivariant Fox-Milnor condition}.

\begin{thm}\label{thm:fox_milnor}
Let $K$ be a strongly invertible knot and let $\Delta_K(t)$ be its Alexander polynomial, normalized so that $\Delta_K(t)=\Delta_K(t^{-1})$ and $\Delta_K(1)=1$. If $K$ is equivariantly slice then $\Delta_K(t)$ is a square.
\end{thm}
\begin{proof}
Suppose $\Delta_K(t)$ does not satisfy the condition above. It follows that there exists an irreducible polynomial $g(t)\neq t$ appearing with an odd power in the prime decomposition of $\Delta_K(t)$. If $g(t^{-1})=\pm t^n g(t)$ for some $n$, then one can easily see that $\Delta_K(t)$ does not satisfy the (non-equivariant) Fox-Milnor condition, therefore $K$ is not even slice.
On the other hand, if $g(t)\neq\pm t^n g(t)$, we know that $g(-t)$ appears in the factorization of $\Delta_K(t)$ with the same exponent $2k+1$, since $\Delta_K$ is symmetric.
Let $p(s)=\delta(g(t)g(t^{-1}))(s)$. We prove that $p(s)$ is irreducible. Suppose that $p(s)=p_1(s)p_2(s)$ is a non-trivial decomposition. Then we would have a factorization of $g(t)g(t^{-1})=\delta^{-1}(p_1)(t)\delta^{-1}(p_2)(t)$, where $\delta^{-1}(p_i)(t)$ are non-trivial symmetric polynomials. But this is not possible since $g(t)$ and $g(t^{-1})$ are non-symmetric and irreducible. 
It follows that $p(s)$ appears in the irreducible factorization of $\delta_K(s)$ with exponent $2k+1$.
Hence the class of $K$ in $\widetilde{\G}_r^\Q$ is non-trivial (with respect to any choice of direction), since its projection on the summand $W(\Q[s]/(p(s)))$ is represented by a form of odd rank. Therefore $K$ is not equivariantly slice.
\end{proof}

Observe that Theorem \ref{thm:fox_milnor} highlights an important difference between equivariant and non-equivariant algebraic concordance. In fact, given a knot $K$, its non-equivariant algebraic concordance class splits into components depending on the irreducible and symmetric factors of $\Delta_K(t)$, while non-symmetric factors of $\Delta_K(t)$ do not contribute.
On the other hand, if $K$ is a strongly invertible knot, its class in $\widetilde{\G}_r^\Q$ splits into components depending on all irreducible factors of $\delta_K(s)$, and these factors can correspond to irreducible non-symmetric factors of $\Delta_K(t)$.

\section{New equivariant signatures}\label{sect:new_equiv_sign}
In this section, we introduce new equivariant signatures, dependent on a parameter $\lambda\in\R$. Then we clarify the relation between $\widetilde{\sigma}_\lambda$, the Levine-Tristram signatures and the equivariant signature defined in \cite{alfieri2021strongly}.
Furthermore, we give an analysis of the discontinuities of these equivariant signatures similar to the one in \cite{matumoto1977signature}, and we use the signature jumps to obtain lower bounds on the equivariant slice genus.

From now on, we always assume to work with coefficients in $\R$ if not mentioned otherwise. In particular, when we refer to equivariant Seifert form or symmetric structure, we always implicitly consider their natural extensions given by tensoring with $\R$.
\begin{defn}\label{def:sign_and_jumps}
Let $(A,P)$ be an equivariant Seifert form. Given $\lambda\in\R$ consider the hermitian form
$$
A_{\lambda}=\frac{(A+A^t)}{2}((1-\lambda)I-(1+\lambda)P)+i(A-A^t).
$$
We define the \emph{equivariant signature} in $\lambda$ as $$\widetilde{\sigma}_\lambda(A,P)=\lim_{\varepsilon\to 0^+}\frac{\sigma(A_{\lambda+\epsilon})+\sigma(A_{\lambda-\epsilon})}{2},$$
and the \emph{equivariant signature jump} in $\lambda$ as $$\widetilde{J}_\lambda(A,P)=\lim_{\varepsilon\to 0^+}\frac{\sigma(A_{\lambda+\epsilon})-\sigma(A_{\lambda-\epsilon})}{2}.$$
\end{defn}

It is immediate to see that $\widetilde{\sigma}_\lambda$ and $\widetilde{J}_\lambda$ define homomorphisms
$$
\widetilde{\G}_r\longrightarrow\Z.
$$
In the following, given a directed strongly invertible knot $K$, we will denote by $\widetilde{\sigma}_\lambda(K)$ and $\widetilde{J}_\lambda(K)$ the compositions of the homomorphisms above with the map $\C\longrightarrow\widetilde{\G}_r^\Z$.
In other words, $\widetilde{\sigma}_\lambda(K)=\widetilde{\sigma}_\lambda(A,P)$ and $\widetilde{J}_\lambda(A,P)$, where $(A,P)$ is the equivariant Seifert form given by an equivariant Seifert surface of type $0$ for $K$.

\begin{remark}\label{remark:signature_meaning}
Up to a base change, we can always suppose that
$$
P=\begin{pmatrix}
    \id&0\\
    0&-\id
\end{pmatrix}.
$$
In such a base, we have that
$$
A=\begin{pmatrix}
    B&C\\
    -C^t&D
\end{pmatrix}.
$$
and hence
$$
A_\lambda=2\begin{pmatrix}
    -\lambda B&iC\\
    -iC^t&D
\end{pmatrix}
$$
which is easily seen to be congruent to
$$
2\begin{pmatrix}
    -\lambda B-CD^{-1}C^t&0\\
    0&D
\end{pmatrix}.
$$
If we denote by $(V,\beta,S)$ the symmetric structure associated with $(A,P)$, we have that $2B$ represents $\beta$, while $S$ is given by $-B^{-1}C^tD^{-1}C$.
Therefore, it is easy to see that $-2\lambda B-2C^tD^{-1}C$ represents the symmetric bilinear form over $V$ given by 
$$
\beta_{S-\lambda}:V\times V\longrightarrow\R
$$
$$
(x,y)\longmapsto \beta((S-\lambda)x,y).
$$
It follows that $A_\lambda$ is non-degenerate for every $\lambda\in\R$ except when $\lambda$ is a root of the characteristic polynomial of $S$.
When $A_\lambda$ is non-degenerate then $\widetilde{J}_\lambda(A,P)=0$ and $\widetilde{\sigma}_\lambda(A,P)$ is equal to the signature of $A_\lambda$.
Moreover, for every $\lambda\in\R$ we have that
$$
\widetilde{J}_\lambda(A,P)=\lim_{\varepsilon\to 0^+}\frac{\sigma(\beta_{S-\lambda-\epsilon})-\sigma(\beta_{S-\lambda+\epsilon})}{2}.
$$

Finally, observe that
$$
\lim_{\lambda\to-\infty}\widetilde{\sigma}_\lambda(A,P)=\sigma(B)+\sigma(D)\qquad\text{and}\qquad\lim_{\lambda\to+\infty}\widetilde{\sigma}_\lambda(A,P)=\sigma(D)-\sigma(B).
$$
Therefore, given a directed strongly invertible knot $K$, we get from Theorem \ref{equivalent_definition} that
$$
\lim_{\lambda\to-\infty}\widetilde{\sigma}_\lambda(K)=\sigma(K)\qquad\text{and}\qquad\lim_{\lambda\to+\infty}\widetilde{\sigma}_\lambda(K)=\widetilde{\sigma}(K).
$$
\end{remark}

\begin{prop}\label{prop:levine_tristram}
Let $K$ be a directed strongly invertible knot. Then for any $\lambda<0$ and $\omega\in S^1$ such that $\lambda(\omega-1)^2=(\omega+1)^2$ we have
$$
\widetilde{\sigma}_\lambda(K)=\sigma_\omega(K),
$$
\end{prop}

\begin{proof}
Let $(A,P)$ be an equivariant Seifert form associated with $K$. As in Remark \ref{remark:signature_meaning}, we can suppose that
\begin{align*}
A=\begin{pmatrix}
    B&C\\
    -C^t&D
\end{pmatrix}\quad&\quad
P=\begin{pmatrix}
    \id&0\\
    0&-\id
\end{pmatrix}\quad&\quad
A_{\lambda}=2\begin{pmatrix}
    -\lambda B&iC\\
    -iC^t&D
\end{pmatrix}.
\end{align*}
Since $\lambda<0$, we have that $\sqrt{-\lambda}\in\R$ and hence $A_{\lambda}$ is congruent to
$$2\begin{pmatrix}
    B&i\sqrt{-\lambda}C\\
    -i\sqrt{-\lambda}C^t&D
\end{pmatrix}=(A+A^t)+i\sqrt{-\lambda}(A-A^t).$$
Finally, observe that $(A+A^t)+i\sqrt{-\lambda}(A-A^t)$ is a positive multiple of $A_\omega=(1-\omega)A+(1-\overline{\omega})A^t$ for $\omega=-\frac{1\pm i\sqrt{-\lambda}}{1\mp i\sqrt{-\lambda}}$. Therefore $\widetilde{\sigma}_{\lambda}(K)=\sigma(A_\omega)=\sigma_\omega(K)$. Finally observe that for a fixed $\lambda<0$, the solution to the equation $\lambda(\omega-1)^2=(\omega+1)^2$ are exactly $\omega=-\frac{1\pm i\sqrt{-\lambda}}{1\mp i\sqrt{-\lambda}}$.
\end{proof}

Let $(V,\beta,S)$ be a symmetric structure. Given $\lambda\in\R$, we denote by $V^\lambda=\bigcup_{N>0}\ker((S-\lambda)^N)$.
Similarly to Lemma \ref{lemma:pq-orth} it is easy to check that if $\lambda\neq \mu$ the $V^\lambda$ and $V^\mu$ are $\beta$-orthogonal.

\begin{lemma}\label{lemma:signature_variation}
Let $(V,\beta,S)$ be a symmetric structure. Suppose that $\lambda\in\R$ is not a root of the characteristic polynomial of $S$. Then for every $\alpha\in\R$ we have $$\sigma(V^\alpha,\beta_{S-\lambda})=\sign(\alpha-\lambda)\cdot\sigma(V^\lambda,\beta).$$
\end{lemma}
\begin{proof}
By a slight abuse of notation, we denote the restriction of $S$ and $\beta$ to $V^\alpha$ again by $S$ and $\beta$.
For $t\in[0,1]$ let $S_t=(1-t)S+t\alpha\id$ and $\beta_{S_t-\lambda}(x,y)=\beta((S_t-\lambda)x,y)$. Observe that the only eigenvalue of $S_t$ is $\alpha$ for every $t\in[0,1]$, therefore we have that $S_t-\lambda\id$ is non-singular and $\beta_{S_t-\lambda}$ is non-degenerate for every $t$.
It follows that the signature is constant in $t$, and hence $$\sigma(V^\alpha,\beta_{S-\lambda})=\sigma(V^\alpha,(\alpha-\lambda)\beta)=\sign(\alpha-\lambda)\sigma(V^\alpha,\beta)$$
by evaluating at $t=0$ and $t=1$.
\end{proof}
\begin{prop}\label{prop:jump_eigenspace}
Let $(A,P)$ be an equivariant Seifert form, and let $(V,\beta,S)$ be the associated symmetric structure.
Then for every $\lambda\in\R$ we have that $\widetilde{J}_\lambda(A,P)$ is equal to $-\sigma(V^\lambda,\beta)$.
\end{prop}
\begin{proof}
It follows from Remark \ref{remark:signature_meaning} that is sufficient to study how the signature of $\beta_{S-\lambda}$ varies in $\lambda$.
Let $W\subset V$ the $\beta$-orthogonal to $\bigoplus_{\lambda\in\R}V^\lambda$. Observe that $W$ is $S$-invariant and that $S_{|W}$ has no real eigenvalues.
Therefore, the signature of the restriction of $\beta_{S-\lambda}$ to $W$ is constant in $\lambda$.

Since if $\mu_1\neq\mu_2$ then $V^{\mu_1}$ and $V^{\mu_2}$ are orthogonal with respect to $\beta$ and hence with respect to $\beta_{S-\lambda}$, it is sufficient to consider the variation of the signature separately on every $V^\alpha$.

Let now $\epsilon>0$ be small enough so that $\lambda$ is the only eigenvalue of $S$ in $[\lambda-\epsilon,\lambda+\epsilon]$.

Then by Lemma \ref{lemma:signature_variation} we have that $\sigma(V^\mu,\beta_{S-\lambda+\epsilon})=\sigma(V^\mu,\beta_{S-\lambda-\epsilon})$ for every $\mu\neq\lambda$, while
$$
\sigma(V^\lambda,\beta_{S-\lambda+\epsilon})=-\sigma(V^\lambda,\beta_{S-\lambda-\epsilon})=-\sigma(V^\lambda,\beta).
$$

Since the signature of $\beta_{S-\lambda}$ varies only on the summand $V^\lambda$, we get
$$
\widetilde{J}_\lambda(A,P)=-\sigma(V^\lambda,\beta_{|V^\lambda}).
$$
\end{proof}

\begin{cor}\label{cor:indep_jump}
Let $K$ be a directed strongly invertible knot, and let $(A,P)$ be an equivariant Seifert form obtained from an equivariant Seifert surface for $K$ of any type. Then for every $\lambda\in\R$, $\lambda\neq 1$ we have that $\widetilde{J}_\lambda(K)=\widetilde{J}_\lambda(A,P)$, i.e. for $\lambda\neq 1$, $\widetilde{J}_\lambda$ does not depend on the type of the equivariant Seifert surface.
\end{cor}
\begin{proof}
Let $(A,P)$ be the equivariant Seifert form given by any equivariant Seifert surface $F$ for $K$, and let $(V,\beta,S)$ be the associated symmetric structure. Recall that we can obtain a surface of type $0$ for $K$ by performing a boundary connected sum of $F$ with the type $-1$ surface $G$ (or its mirror image) as Lemma \ref{stab_surface}. Observe that the equivariant Seifert form given by $G$ is
\begin{align*}
    A_G=\begin{pmatrix}
        0&1\\
        0&0
    \end{pmatrix}\quad&\quad
    P_G=\begin{pmatrix}
        0&1\\
        1&0
    \end{pmatrix},
\end{align*}
and hence the associated symmetric structure is
$$
\beta_G=(1)\quad\quad S_G=(1).
$$
It follows that adding some copies of $(A_G,P_G)$ to $(A,P)$ does not modify the $\lambda$-eigenspace of $S$, for $\lambda\neq1$.
Therefore, by Proposition \ref{prop:jump_eigenspace} we can compute $\widetilde{J}_\lambda(K)$ with any equivariant Seifert surface, for $\lambda\neq1$.
\end{proof}

\begin{thm}\label{thm:genus_bound_jump}
Given a directed strongly invertible knot $K$, for every $\lambda\in\R$, $\lambda\neq 1$ we have that
$$\widetilde{g}_4(K)\geq\frac{|\widetilde{J}_\lambda(K)|}{4}.$$
\end{thm}
\begin{proof}
Let $\mathcal{S}=(A,P,h,\widetilde{\lk})$ be the equivariant Seifert system associated with an equivariant Seifert surface $F$ for $K$ of any type.
It follows easily from Definitions \ref{def:alg_complx} and \ref{def:sign_and_jumps} that for every $\lambda\in\R$
$$
ac(\mathcal{S})\geq |\widetilde{\sigma}_{\lambda}(A,P)|/2,
$$
and therefore
$$
ac(\mathcal{S})\geq |\widetilde{J}_{\lambda}(A,P)|/2.
$$
From Corollary \ref{cor:indep_jump} we know that for $\lambda\neq1$ the left-hand side in the inequality above does not depend on the type of the surface $F$. Therefore using Theorem \ref{lower_bound_genus} we get
$$\widetilde{g}_4(K)\geq \frac{sc(K)}{2}\geq\frac{|\widetilde{J}_\lambda(K)|}{4}.$$
\end{proof}

\begin{remark}\label{remark:transf_roots}
Observe that the roots of $\delta_K(s)$ correspond to pairs of roots of $\Delta_K(t)$.
First of all, notice that $-1,0,1$ cannot be roots of $\Delta_K(t)$.
It follows from Definition \ref{def:poly_transf} that if $z\in\mathbb{C}\setminus\{-1,0,1\}$ is a root of $\Delta_K(t)$ then $\mu(z)$ is a root of $\delta_K(s)$, where
$$
\mu:\mathbb{C}\setminus\{-1,0,1\}\longrightarrow\mathbb{C}\setminus\{0,1\}
$$
$$
z\longmapsto \left(\frac{z-1}{z+1}\right)^2.
$$
It is not difficult to see that $\mu(z_1)=\mu(z_2)$ if and only if $z_1=(z_2)^{-1}$ and that $\mu(z)\in\R$ if and only if $z\in\R$ or $z\in S^1$. Therefore, if $\{\lambda_1,\lambda_1^{-1},\dots,\lambda_d,\lambda_d^{-1}\}$ is the set of roots of $\Delta_K(t)$ (which come in pairs since the polynomial is symmetric), we get that $\{\mu(\lambda_1),\dots,\mu(\lambda_d)\}$ is the set of roots of $\delta_K(s)$, and this correspondence preserves the multiplicity. In particular, real negative roots of $\delta_K(s)$ correspond to unitary roots of $\Delta_K(t)$, while real positive roots of $\delta_K(s)$ correspond to real roots of $\Delta_K(t)$.
\end{remark}

Using Theorem \ref{thm:genus_bound_jump} we obtain the following result, which can be seen as a generalization of \cite[Theorem 5.7]{miller_powell}.

\begin{thm}\label{thm:genus_bound_root}
Let $K$ and $J$ be directed strongly invertible knots. Suppose that there exists $\Lambda\in\R\cup S^1$ such that
\begin{itemize}
    \item $\Lambda$ is a root of $\Delta_K(t)$ with odd multiplicity,
    \item $\Lambda$ is not a root of $\Delta_J(t)$.
\end{itemize}
Then for every $n,m\in\Z$, we have that $\widetilde{g}_4(K^n\widetilde{\#}J^m)\geq|n|/4$, where $K^n\widetilde{\#}J^m$ is the connected sum of $n$ copies of $K$ and $m$ copies of $J$ taken in any order.
\end{thm}
\begin{proof}
Let $(V_K,\beta_K,S_K)$ and $(V_J,\beta_J,S_J)$ be any symmetric structures associated with $K$ and $J$ respectively, and let $\lambda=\mu(\Lambda)$. From Remark \ref{remark:transf_roots} we know that $\lambda$ is a real root of $\delta_K(s)$ with odd multiplicity. Therefore $\dim(V^\lambda_K)$ is odd and $|\widetilde{J}_\lambda(K)|\geq 1$ by Proposition \ref{prop:jump_eigenspace}.
Viceversa, since $\lambda$ is not a root of $\delta_J(s)$, we have that $\dim(V^\lambda_J)=0$  and hence $\widetilde{J}_\lambda(J)=0$.
It follows that $|\widetilde{J}_\lambda(K^n\widetilde{\#}J^m)|\geq |n|$, and hence $\widetilde{g}_4(K^n\widetilde{\#}J^m)\geq |n|/4$ by Theorem \ref{thm:genus_bound_jump}.
\end{proof}

We conclude this section with an example of application of Theorem \ref{thm:genus_bound_root}.

\begin{ex}\label{ex:alexander_roots}
Let $\{K_i\}_{i\in I}$ be a family of (algebraically) slice strongly invertible knots and $\{\lambda_i\}_{i\in I}\subset \R$ such that:
\begin{enumerate}
    \item $\lambda_i$ is a root of $\Delta_{K_i}(t)$ with odd multiplicity,
    \item $\Delta_{K_j}(\lambda_i)\neq 0$ for $i\neq j$.
\end{enumerate}
Endow each $K_i$ with any choice of direction and consider the subgroup $G_{I}\subset\C$ generated by $\{K_i\}_{i\in I}$.
It follows immediately from Theorem \ref{thm:genus_bound_root} that the image of $G_I$ in $\widetilde{\G}_r^{\Q}$ has rank $|I|$.

For example, take $I\subset\Z[t]$ to be an infinite family of irreducible, non-symmetric polynomials of odd degree, such that
\begin{itemize}
    \item $f(1)=1$ for all $f\in I$,
    \item $f(t)\neq t^{\deg g}g(t^{-1})$ for all $f,g\in I$.
\end{itemize}
From \cite{sakai1983polynomials} we know that for all $f\in I$ there exists a strongly invertible knot $K_f$ with Alexander polynomial $\Delta_{K_f}(t)=f(t)f(t^{-1})$. In particular, each $K_f$ is algebraically slice and the family $\{K_f\}_{f\in I}$ satisfies the conditions above (since each $f$ has at least one real root with multiplicity one), therefore it generates a subgroup of $\C$ of infinite rank.
\end{ex}

\begin{remark}\label{rem:lin_indep}
Let $G_I\subset\C$ be a subgroup defined as in Example \ref{ex:alexander_roots}, with $I$ a countable infinite set.
Observe that the equivariant signature defined by Alfieri and Boyle takes values in $\Z$, hence it vanishes on a subgroup $H_I\subset G_I$ such that $G_I/H_I$ is either $0$ or $\Z$. Moreover, all Levine-Tristram signatures vanish on $H_I$, since by hypothesis is spanned by algebraically slice knots. However, using the equivariant signature jumps arguing as above, we have that $H_I$ surjects onto $\Z^\infty$. In particular, this shows that the $\{\widetilde{J}_\lambda\}_{\lambda>0}$ are actually new invariants, independent from Levine-Tristram and Alfieri-Boyle signatures. For other concrete examples regarding the independence of $\widetilde{J}_\lambda$ from Levine-Tristram signatures, see the Appendix \ref{apx:2_bridge}.
\end{remark}

\section{Appendix: equivariant slice genus of 2-bridge knots}\label{apx:2_bridge}
Recall that every $2$-bridge knot is strongly invertible, see \cite{sakuma}.
We already know from \cite{diprisa_framba} that any given $2$-bridge knot is not equivariantly slice and that has infinite order in $\C$\footnote{This result is stated in the smooth category in \cite{diprisa_framba}. However, all the arguments can be adapted to work in the topological category.}, independently of the choice of strong inversion and direction.

In this appendix\footnote{We used \cite{knotinfo} to gather the Alexander polynomial, Levine-Tristram signature function and topological concordance order of all knots considered in the appendix.} we apply the results of Section \ref{sect:new_equiv_sign} to study the equivariant slice genus of $2$-bridge knots.

\begin{defn}
Let $K$ be a directed strongly invertible knot. We denote the \emph{maximal signature jump} of $K$ by
$$\widetilde{J}(K)=\sup_{\lambda\neq 1}|\widetilde{J}_\lambda(K)|.$$
\end{defn}

Let $\mathcal{F}$ be the family of $2$-bridge knots with crossing number less or equal to 12 and with (averaged) Levine-Tristram signature function $\sigma_\omega$ identically zero. 
For every knot $K$ in $\mathcal{F}$ we report in Table \ref{table:2_bridge} its
\begin{itemize}
    \item name,
    \item $2$-bridge notation $p/q$,
    \item order in $\mathcal{C}$, denoted by $Ord$,
    \item maximal signature jump $\widetilde{J}$\footnote{see Lemma \ref{lemma:max_sign_jump}}.
\end{itemize}

The purpose of Table \ref{table:2_bridge} is to show that given a directed strongly invertible knot $K$, the equivariant signature jumps can be used to prove that the (topological) equivariant slice genus of $K^n$ grows linearly in $|n|$ even when the Levine-Tristram signature function vanishes.

Given a knot $K$ in the family $\mathcal{F}$, it is not possible to obtain a lower bound on the smooth or topologically slice genus of $n\cdot K$, $n\in\Z$ using the Levine-Tristram signatures, since $\sigma_\omega(K)\equiv 0$.

On the other hand, observe from Table \ref{table:2_bridge} that for most of the knots in $\mathcal{F}$, we have $\widetilde{J}(K)=1$. Therefore, we easily get the following lower bound on the topological equivariant slice genus of $K^n$
$$\widetilde{g}_4(K^n)\geq |n|/4,$$
by applying Theorem \ref{thm:genus_bound_jump}.

We rely on the following lemma in order to compute $\widetilde{J}(K)$ for $K$ in $\mathcal{F}$.

\begin{lemma}\label{lemma:max_sign_jump}
For any knot $K$ in $\mathcal{F}$ we have
$$
\widetilde{J}(K)=\begin{cases}
    0\quad\text{if}\quad\Delta_K(t)\quad\text{has no real root},\\
    1\quad\text{if}\quad\Delta_K(t)\quad\text{has at least one real root},
\end{cases}
$$
and in particular $\widetilde{J}(K)$ does not depend on the choice of strong inversion nor of direction on $K$.
\end{lemma}
\begin{proof}
First of all we check that each root of $\Delta_K(t)$ has multiplicity one.
It follows that $\Delta_K(t)$ has no root $z\in S^1$, otherwise we would have a discontinuity of the signature function in $z$ and hence $\sigma_\omega(K)\neq 0$ for some $\omega\in S^1$.
Using Remark \ref{remark:transf_roots} and arguing as in the proof of Theorem \ref{thm:genus_bound_jump} one can see that in this case $\widetilde{J}(K)$ does not depend on the choice of strong inversion nor of direction on $K$, and that
$$
\widetilde{J}(K)=\begin{cases}
    0\quad\text{if}\quad\Delta_K(t)\quad\text{has no real root},\\
    1\quad\text{if}\quad\Delta_K(t)\quad\text{has at least one real root}.
\end{cases}
$$
\end{proof}

\begin{remark}\label{remark:2_bridge_not_square}
Observe that since each root of $\Delta_K(t)$ has multiplicity one for $K$ in $\mathcal{F}$, $\Delta_K(t)$ cannot be a square. Hence, using Theorem \ref{thm:fox_milnor} we recover the result in \cite{diprisa_framba}: none of the knots in $\mathcal{F}$ is equivariantly slice.
\end{remark}

\begin{remark}\label{remark:big_equiv_genus}
In \cite{dai_mallick_stoffregen} the authors provide the first examples of strongly invertible knots with smooth equivariant slice genus arbitrarily larger than their smooth slice genus.
Miller and Powell in \cite{miller_powell} find examples of knots on which the quantity
$
\widetilde{g}_4(K)-g_4(K)$
assume arbitrarily large values.

In the family $\mathcal{F}$ we can find knots $K$ with smooth and topological order $\leq 2$ and such that $\widetilde{J}(K)=1$ ($8_9$). Since the smooth and topological slice genus of $K^n$ is bounded, applying Theorem \ref{thm:genus_bound_jump} we get that both $
\widetilde{g}_4(K^n)-g_4(K^n)$
grows linearly in $|n|$ and hence they provide new examples of the phenomenon described above.
The knot $8_9$ is the simplest example of such knots in $\mathcal{F}$ that does not fall into the examples in \cite{dai_mallick_stoffregen,miller_powell}. 
\end{remark}

\newpage
\begin{table}[H]
    \centering
    \caption{}
    \label{table:2_bridge}

\begin{tabular}{c c}

\begin{tabular}{|c|c|c|c|}
Name& $p/q$ & $Ord$ & $\widetilde{J}$\\
\hline\hline
$4_1$   &   $5/2$   &   $2$ &   $1$\\\hline
$6_1$   &   $9/7$   &   $1$ &   $1$\\\hline
$6_3$   &   $13/5$   &   $2$ &   $0$\\\hline
$7_7$   &   $21/8$   &   $\infty$ &   $0$\\\hline
$8_1$   &   $13/11$   &   $\infty$ &   $1$\\\hline
$8_3$   &   $17/4$   &   $2$ &   $1$\\\hline
$8_8$   &   $25/9$   &   $1$ &   $0$\\\hline
$8_9$   &   $25/7$   &   $1$ &   $1$\\\hline
$8_{12}$   &   $29/12$   &   $2$ &   $1$\\\hline
$8_{13}$   &   $29/11$   &   $\infty$ &   $0$\\\hline
$9_{14}$   &   $37/14$   &   $\infty$ &   $0$\\\hline
$9_{19}$   &   $41/16$   &   $\infty$ &   $0$\\\hline
$9_{27}$   &   $49/19$   &   $1$ &   $1$\\\hline
$10_1$   &   $17/15$   &    $\infty$ &   $1$\\\hline
$10_3$   &   $25/6$   &   $1$ &   $1$\\\hline
$10_{10}$   &   $45/17$   &     $\infty$ &   $0$\\\hline
$10_{13}$   &   $53/22$   &     $\infty$ &   $1$\\\hline
$10_{17}$   &   $41/9$   &   $2$ &   $0$\\\hline
$10_{22}$   &   $49/36$   &    $1$ &   $1$\\\hline
$10_{26}$   &   $61/44$   &    $\infty$ &   $1$\\\hline
$10_{28}$   &   $53/19$   &     $\infty$ &   $0$\\\hline
$10_{31}$   &   $57/25$   &     $\infty$ &   $0$\\\hline
$10_{33}$   &   $65/18$   &    $2$ &   $0$\\\hline
$10_{34}$   &   $37/13$   &     $\infty$ &   $0$\\\hline
$10_{35}$   &   $49/20$   &    $1$ &   $1$\\\hline
$10_{37}$   &   $53/23$   &   $2$ &   $0$\\\hline
$10_{42}$   &   $81/31$   &   $1$ &   $1$\\\hline
$10_{43}$   &   $73/27$   &    $2$ &   $1$\\\hline
$10_{45}$   &   $89/34$   &    $2$ &   $1$\\\hline
$11a_{13}$   &   $61/28$   &     $\infty$ &   $0$\\\hline
$11a_{84}$   &   $101/57$   &   $\infty$ &   $1$\\\hline
$11a_{91}$   &   $129/50$   &     $\infty$ &   $1$\\\hline
$11a_{96}$   &   $121/50$   &    $1$ &   $1$\\\hline
$11a_{98}$   &   $77/18$   &    $\infty$ &   $0$\\\hline
$11a_{110}$   &   $97/35$   &     $\infty$ &   $1$\\\hline
$11a_{119}$   &   $77/34$   &     $\infty$ &   $0$\\\hline
$11a_{180}$   &   $89/64$   &     $\infty$ &   $0$\\\hline
$11a_{185}$   &   $109/30$   &     $\infty$ &   $1$\\\hline
$11a_{190}$   &   $85/67$   &    $\infty$ &   $1$\\\hline
$11a_{195}$   &   $53/8$   &     $\infty$ &   $0$\\\hline
$11a_{210}$   &   $73/16$   &   $\infty$ &   $0$\\\hline
$11a_{333}$   &   $65/14$   &    $\infty$ &   $0$\\\hline
$12a_{197}$   &   $69/32$   &    $\infty$ &   $1$\\\hline

\end{tabular}\qquad\qquad&\qquad\qquad

\begin{tabular}{|c|c|c|c|}
Name&$p/q$& $Ord$ & $\widetilde{J}$\\
\hline\hline
$12a_{204}$   &   $173/76$   &   $\infty$ &   $1$\\\hline
$12a_{221}$   &   $169/66$   &    $1$ &   $1$\\\hline
$12a_{243}$   &   $133/60$   &   $\infty$ &   $1$\\\hline
$12a_{303}$   &   $153/64$   &   $\infty$ &   $1$\\\hline
$12a_{307}$   &   $157/69$   &   $\infty$ &   $1$\\\hline
$12a_{385}$   &   $161/66$   &    $\infty$ &   $1$\\\hline
$12a_{425}$   &   $81/37$   &    $1$ &   $0$\\\hline
$12a_{437}$   &   $149/65$   &    $\infty$ &   $1$\\\hline
$12a_{447}$   &   $121/43$   &    $1$ &   $1$\\\hline
$12a_{471}$   &   $85/38$   &     $2$ &   $1$\\\hline
$12a_{477}$   &   $169/70$   &     $1$ &   $1$\\\hline
$12a_{482}$   &   $93/22$   &     $\infty$ &   $1$\\\hline
$12a_{497}$   &   $209/81$   &     $\infty$ &   $0$\\\hline
$12a_{499}$   &   $233/89$   &     $2$ &   $0$\\\hline
$12a_{506}$   &   $185/68$   &     $2$ &   $1$\\\hline
$12a_{510}$   &   $193/81$   &     $2$ &   $1$\\\hline
$12a_{518}$   &   $157/34$   &     $\infty$ &   $1$\\\hline
$12a_{550}$   &   $149/34$   &     $?$ &   $1$\\\hline
$12a_{583}$   &   $161/45$   &     $\infty$ &   $1$\\\hline
$12a_{585}$   &   $181/50$   &     $\infty$ &   $1$\\\hline
$12a_{596}$   &   $81/14$   &      $\infty$ &   $0$\\\hline
$12a_{644}$   &   $113/30$   &    $?$ &   $0$\\\hline
$12a_{650}$   &   $165/46$   &    $\infty$ &   $1$\\\hline
$12a_{690}$   &   $89/20$   &    $?$ &   $1$\\\hline
$12a_{691}$   &   $77/12$   &    $\infty$ &   $1$\\\hline
$12a_{715}$   &   $169/50$   &    $1$ &   $1$\\\hline
$12a_{744}$   &   $61/8$   &     $\infty$ &   $0$\\\hline
$12a_{774}$   &   $89/16$   &    $?$ &   $0$\\\hline
$12a_{792}$   &   $85/24$   &     $\infty$ &   $0$\\\hline
$12a_{803}$   &   $21/2$   &     $\infty$ &   $1$\\\hline
$12a_{1029}$   &   $81/19$   &     $1$ &   $0$\\\hline
$12a_{1034}$   &   $121/32$   &     $1$ &   $0$\\\hline
$12a_{1039}$   &   $137/37$   &     $2$ &   $1$\\\hline
$12a_{1127}$   &   $97/22$   &    $2$ &   $1$\\\hline
$12a_{1129}$   &   $105/23$   &      $\infty$ &   $0$\\\hline
$12a_{1139}$   &   $101/18$   &     $?$ &   $0$\\\hline
$12a_{1140}$   &   $97/18$   &      $?$ &   $1$\\\hline
$12a_{1166}$   &   $33/4$   &     $\infty$ &   $1$\\\hline
$12a_{1273}$   &   $61/11$   &     $2$ &   $1$\\\hline
$12a_{1275}$   &   $149/44$   &    $2$ &   $1$\\\hline
$12a_{1277}$   &   $121/37$   &      $1$ &   $1$\\\hline
$12a_{1281}$   &   $109/33$   &    $2$ &   $1$\\\hline
$12a_{1287}$   &   $37/6$   &    $2$ &   $1$\\\hline

\end{tabular}
\end{tabular}

\end{table}

\section*{Acknowledgements}
I am very grateful to the referees for their suggestions to improve the results presented in earlier drafts and for their useful comments.
I am thankful to Makoto Sakuma for his kind and encouraging comments about this work.

Additionally, I want to thank the Italian National Group for Algebraic and Geometric Structures and their Applications (INdAM-GNSAGA).
\bibliographystyle{alpha} 
\bibliography{refs} 

\end{document}